\documentclass[11pt, reqno]{amsart}

\usepackage[T1]{fontenc}
\usepackage{babel}
\usepackage{lmodern}

\usepackage{latexsym}
\usepackage{amssymb}
\usepackage{mathrsfs}
\usepackage[leqno]{amsmath}
\usepackage{fancybox,color}
\usepackage{enumerate}
\usepackage[latin1]{inputenc}
\usepackage[active]{srcltx}
\usepackage[colorlinks]{hyperref}
\usepackage{hyperref}
\usepackage{eurosym}
\usepackage{url}
\usepackage{xfrac}
\usepackage{color,soul}

\usepackage{fancyhdr}
\usepackage{xfrac}

\usepackage{amsthm}
\newcommand{\kom}[1]{}
\renewcommand{\kom}[1]{{\bf [#1]}}

 \def\1{\raisebox{2pt}{\rm{$\chi$}}}

\newtheorem{theorem}{Theorem}[section]
\newtheorem{corollary}[theorem]{Corollary}
\newtheorem{lemma}[theorem]{Lemma}
\newtheorem{proposition}[theorem]{Proposition}
\newtheorem{definition}[theorem]{Definition}
\expandafter\let\expandafter\oldproof\csname\string\proof\endcsname
\let\oldendproof\endproof
\renewenvironment{proof}[1][\proofname]{%
	\oldproof[\textbf{#1}]%
}{\oldendproof}

\newcommand{\R}{{\mathbb R}}

\newcommand{\N}{{\mathbb N}}

 \newcommand{\eps}{{\varepsilon}}
 \def\1{\raisebox{2pt}{\rm{$\chi$}}}
 
\let\originalleft\left
\let\originalright\right
\renewcommand{\left}{\mathopen{}\mathclose\bgroup\originalleft}
\renewcommand{\right}{\aftergroup\egroup\originalright}

\newcommand{\abs}[1]{\left|#1\right|}
\newcommand{\norm}[1]{\left|\left|#1\right|\right|}
\newcommand{\Rn}{\mathbb{R}^n}
\newcommand{\Rnn}{\mathbb{R}^{n+1}}

\newcommand{\vp}{\varphi}

\def\vint_#1{\mathchoice%
          {\mathop{\kern 0.2em\vrule width 0.6em height 0.69678ex depth -0.58065ex
                  \kern -0.8em \intop}\nolimits_{\kern -0.4em#1}}%
          {\mathop{\kern 0.1em\vrule width 0.5em height 0.69678ex depth -0.60387ex
                  \kern -0.6em \intop}\nolimits_{#1}}%
          {\mathop{\kern 0.1em\vrule width 0.5em height 0.69678ex
              depth -0.60387ex
                  \kern -0.6em \intop}\nolimits_{#1}}%
          {\mathop{\kern 0.1em\vrule width 0.5em height 0.69678ex depth -0.60387ex
                  \kern -0.6em \intop}\nolimits_{#1}}}
\def\vintslides_#1{\mathchoice%
          {\mathop{\kern 0.1em\vrule width 0.5em height 0.697ex depth -0.581ex
                  \kern -0.6em \intop}\nolimits_{\kern -0.4em#1}}%
          {\mathop{\kern 0.1em\vrule width 0.3em height 0.697ex depth -0.604ex
                  \kern -0.4em \intop}\nolimits_{#1}}%
          {\mathop{\kern 0.1em\vrule width 0.3em height 0.697ex depth -0.604ex
                  \kern -0.4em \intop}\nolimits_{#1}}%
          {\mathop{\kern 0.1em\vrule width 0.3em height 0.697ex depth -0.604ex
                  \kern -0.4em \intop}\nolimits_{#1}}}

\newcommand{\aveint}[2]{\mathchoice%
          {\mathop{\kern 0.2em\vrule width 0.6em height 0.69678ex depth -0.58065ex
                  \kern -0.8em \intop}\nolimits_{\kern -0.45em#1}^{#2}}%
          {\mathop{\kern 0.1em\vrule width 0.5em height 0.69678ex depth -0.60387ex
                  \kern -0.6em \intop}\nolimits_{#1}^{#2}}%
          {\mathop{\kern 0.1em\vrule width 0.5em height 0.69678ex depth -0.60387ex
                  \kern -0.6em \intop}\nolimits_{#1}^{#2}}%
          {\mathop{\kern 0.1em\vrule width 0.5em height 0.69678ex depth -0.60387ex
                  \kern -0.6em \intop}\nolimits_{#1}^{#2}}}

\newcommand{\half}{{\frac{1}{2}}}

\newcommand{\Om}{\Omega}

\newcommand{\diam}{\operatorname{diam}}

\renewcommand{\div}{\operatorname{div}}

\newcommand{\sumn}{\sum_{i=1}^{n}}

\newcommand{\Hu}{\overline{H}}
\newcommand{\Hl}{\underline{H}}

\newtheoremstyle{case}{3mm}{-1,5mm}{}{}{}{:}{ }{}
\theoremstyle{case}

\newcommand{\numberthis}{\addtocounter{equation}{1}\tag{\theequation}}

\makeatletter
\newcommand{\leqnomode}{\tagsleft@true}
\newcommand{\reqnomode}{\tagsleft@false}
\makeatother

\DeclareMathOperator{\Tr}{Tr}

\numberwithin{equation}{section}

\usepackage{etoolbox}

\makeatletter
\patchcmd{\@maketitle}
{\if@titlepage \newpage \else}
{\if@titlepage
	\vspace{\baselineskip}
	\else}
{}{}
\makeatother

\let\oldtocsection=\tocsection

\let\oldtocsubsection=\tocsubsection

\let\oldtocsubsubsection=\tocsubsubsection

\renewcommand{\tocsection}[2]{\hspace{0em}\oldtocsection{#1}{#2}}
\renewcommand{\tocsubsection}[2]{\hspace{2em}\oldtocsubsection{#1}{#2}}
\renewcommand{\tocsubsubsection}[2]{\hspace{4em}\oldtocsubsubsection{#1}{#2}}

\calclayout

\newcommand{\overbar}[1]{\mkern 1.5mu\overline{\mkern-1.5mu#1\mkern-1.5mu}\mkern 1.5mu}

\title{Boundary regularity for a general nonlinear parabolic equation in non-divergence form}

\author{Tapio Kurkinen}
\email{tapio.j.kurkinen@jyu.fi}

\date{\today}


\begin{document}

	\begin{abstract}
		We characterize regular boundary points in terms of a barrier family for a general form of a parabolic equation that generalizes both the standard parabolic $p$-Laplace equation and the normalized version arising from stochastic game theory. Using this result we prove geometric conditions that ensure regularity by constructing suitable barrier families. We also prove that when $q<2$, a single barrier does not suffice to guarantee regularity.
	\end{abstract}
	\maketitle
\section{Introduction}
\label{sec:intro}
We examine the boundary regularity for the following general non-divergence form version of the nonlinear parabolic equation
\begin{equation}
	\label{eq:rgnppar}
	\partial_t u=\abs{\nabla u}^{q-p}\div\left(\abs{\nabla u}^{p-2}\nabla u\right)=\abs{\nabla u}^{q-2}(\Delta u + (p-2)\Delta_\infty^Nu),
\end{equation}
where $q>1$ and $p>1$. When $q=p$, this reduces to the usual $p$-parabolic equation and when $q=2$, we get the normalized $p$-parabolic equation arising from stochastic tug-of-war games.

A boundary point is called \textit{regular} with respect to a partial differential equation if all solutions to the Dirichlet problem with continuous boundary values attain their boundary values continuously. Thus a given Dirichlet problem in a set is solvable in the classical sense if and only if all boundary points of the set are regular. Our main result is that the existence of a family of barrier functions at a point is equivalent to that point being regular. There are geometric conditions that imply the existence of barrier families and thus also imply boundary regularity by this characterization. We also show that the existence of a single barrier is not enough when $q<2$. This problem remains open when $q>2$. A key idea we use in the proofs is a radial connection to the $p$-parabolic case. The radial solutions of \eqref{eq:rgnppar} solve in a suitable sense a weighted one dimensional $q$-parabolic equation as proven by Parviainen and Vázquez \cite{Parviainen2020}. Thus the radial barrier functions and sets used in the proofs are similar between these two equations. Using barrier constructions, we prove that the exterior ball condition and two other geometric conditions guarantee boundary regularity. 

Parabolic boundary regularity is delicate. Petrovski\u{\i} criterion for the one-dimensional heat equation, presented in \cite{Petrovksii1934} and proven in \cite{Petrovksii1935}, shows that a boundary point that is regular for the equation
\begin{equation*}
	\partial_{t}u=\Delta u
\end{equation*} turns out to be irregular for the multiplied equation
\begin{equation*}
	2\partial_{t}u=\Delta u.
\end{equation*} However surprisingly boundary points remain regular for all multiplied $p$-parabolic equation when $p\not=2$ as proven in \cite{Bjorn2015}. We prove that similar phenomenon happens for equation \eqref{eq:rgnppar} when $q\not=2$. When $q=2$, any multiple of a solution is also a solution and thus existence of a barrier implies the existence of a barrier family. It would seem that this might suggest that one barrier is not enough when $q>2$, but we have not yet found a counterexample or a proof for the contrary.

Characterizing regular boundary points for different equations has a long history. One of the most celebrated of these is the original Wiener criterion proven by Norbert Wiener in 1924 \cite{Wiener1924} for the Laplace equation. Wiener type criterion for the heat equation was proven by Evans and Gariepy \cite{Evans1982} but remains open for the usual $p$-parabolic equation. For equations of $p$-parabolic type, the approach using barrier functions has proven fruitful. These seem to date back to Poincaré \cite{Poincare1890} but were named by Lebesgue in \cite{Lebesgue1924} where he characterizes regularity for the Laplace equation using barriers. Granlund, Lindqvist, and Martio extended the Perron method and established the barrier approach to the elliptic $p$-Laplacian in their paper \cite{Granlund1986} and developed it in various papers. An overview of the elliptic results is given in the book by Heinonen, Kilpeläinen, and Martio \cite{Heinonen2006}. The theory for the $p$-parabolic case was initiated by Kilpeläinen and Lindqvist \cite{Kilpelainen1996} where they established the parabolic Perron method and suggested a barrier approach. Björn, Björn, Gianazza, and Parviainen characterized boundary regularity using a family of barriers in \cite{Bjorn2015} and proved that a single barrier does not suffice for singular exponents in \cite{Bjorn2017}. A single barrier turns out to be enough for the normalized $p$-parabolic equation as shown by Bannerjee and Garofalo \cite{Banerjee2014}. Björn, Björn, and Parviainen proved a tusk condition and a Petrovski\u{\i} criterion for this equation in \cite{Bjorn2019}. Boundary regularity for the porous medium equation was examined in \cite{Bjorn2018}.

Since equation \eqref{eq:rgnppar} is in non-divergence form except in the special case, we will use the concept of viscosity solutions. A suitable definition taking account potential singularities was established by Ohnuma and Sato in \cite{Ohnuma1997}. The normalized $p$-parabolic equation arises from game theory which was first examined in the parabolic setting in \cite{Manfredi2010}. This problem has attained recent interest for example in \cite{Jin2017}, \cite{Hoeg2019}, \cite{Dong2020} and \cite{Andrade2022} in addition to already mentioned \cite{Bjorn2019}.
The general form of \eqref{eq:rgnppar} has been examined for example in
\cite{Imbert2019},\cite{Parviainen2020},\cite{Kurkinen2022} and \cite{Kurkinen2023}.

The structure of the paper is as follows. In Section \ref{sec:pre} we present a suitable definition of viscosity solutions to equation \eqref{eq:rgnppar} that takes into account the potential singularity of the equation and state some known results we need later. In Section \ref{sec:ecomp} we present and prove an elliptic-type comparison principle for equation \eqref{eq:rgnppar}. In Section \ref{sec:perron} we define Perron solutions and prove some basic properties. Sections \ref{sec:barrier} and \ref{sec:count} consist of defining boundary regularity and barriers and proving our main result. We also prove that regularity is a local property and show by a counterexample that a single barrier is not enough to prove regularity. In Section \ref{sec:ext}, we establish the exterior ball condition and a few other geometric conditions by constructing suitable barrier families. In Section \ref{sec:app}, we analyze the connection of our definition for a barrier and the ones appearing in the literature for other equations.

\section{Prerequisites}
\label{sec:pre}
In this paper, we denote the dimension by $n$ and let $\Om\subset\Rn$ and $\Theta\subset\Rnn$ be open and bounded sets. Denote $\Om_T=\Om\times(0,T)$ and $\Om_{t_1,t_2}=\Om\times(t_1,t_2)$ the spacetime cylinders and a parabolic boundary by
\begin{equation*}
	\partial_{p}\Om_{t_1,t_2}=\left(\Om\times\{t_1\}\right)\cup\left({\partial\Om\times[t_1,t_2]}\right).
\end{equation*}
We denote the Euclidean ball of radius $r>0$ centered at $x_0\in\Rn$ by $B_r(x_0)$ and $Q_r$ denotes the scaled cylinder
\begin{equation*}
	Q_r=B_r(0)\times(-r^2,0].
\end{equation*} For $\xi\in\Rnn$ and $A\subset \Rnn$, we denote
\begin{equation*}
	\xi+A=\left\{\xi+a\mid a\in A\right\}.
\end{equation*}

When $\nabla u\not =0$, we denote
\begin{equation}
	\Delta_p^qu=\abs{\nabla u}^{q-p}\div\left(\abs{\nabla u}^{p-2}\nabla u\right)=\abs{\nabla u}^{q-2}(\Delta u + (p-2)\Delta_\infty^Nu),
\end{equation}
where $p>1$ and $q>1$ are real parameters and the normalized or game theoretic infinity Laplace operator is given by
\begin{equation*}
	\Delta_\infty^Nu=\sum_{i,j=1}^{n}\frac{\partial_{x_i}u \, \partial_{x_j}u \, \partial_{x_i x_j}u}{\abs{\nabla u}^2}.
\end{equation*}
Thus equation \eqref{eq:rgnppar} can be written as
\begin{equation*}
	\partial_tu=\Delta_p^qu.
\end{equation*}

 Denote
\begin{equation}
	\label{eq:gnnpparf}
	F(\eta,X)=\abs{\eta}^{q-2}\Tr\left(X+(p-2)\frac{\eta\otimes \eta}{\abs{\eta}^2}X\right)
\end{equation}
where $(a\otimes b)_{ij}=a_ib_j$, so that
\begin{align*}
	F(\nabla u,D^2u)&=\abs{\nabla u}^{q-2}(\Delta u+(p-2)\Delta_\infty^Nu)=\Delta_p^qu
\end{align*}
whenever $\nabla u\not=0$. This $F$ is \textit{degenerate elliptic}, meaning that
\begin{equation*}
	F(\eta,X)\leq F(\eta,Y)
\end{equation*} for all $\eta\in\Rn\setminus\{0\}$ and $X\leq Y$. 

We will need to restrict the class of test functions in the definition of a viscosity solution to deal with the singularity of the equation near critical points.
Let $\mathcal{F}(F)$ be the set of functions $f\in C^2([0,\infty))$ such that
\begin{equation*}
	f(0)=f'(0)=f''(0)=0 \text{ and } f''(r)>0 \text{ for all }r>0,
\end{equation*}
and also require that for $g(x)=f(\abs{x})$, it holds that
\begin{equation*}
	\lim_{\substack{x\to0\\x\not=0}}F(\nabla g(x),D^2g(x))=0.
\end{equation*}
This set $\mathcal{F}(F)$ is never empty because it is easy to see that $f(r)=r^\beta\in\mathcal{F}(F)$ for any $\beta>\max\left(\frac{q}{q-1},2\right)$. Note also that if $f\in\mathcal{F}(F)$, then $\lambda f\in\mathcal{F}(F)$ for all $\lambda>0$.

Additionally define the set
\begin{equation*}
	\Sigma=\{\sigma\in C^1(\R)\mid \sigma \text{ is even}, \sigma(0)=\sigma'(0)=0, \text{ and }\sigma(r)>0 \text{ for all }r>0\}.
\end{equation*}
We use $\mathcal{F}(F)$ and $\Sigma$ to define an admissible set of test functions for viscosity solutions.
\begin{definition}\sloppy
	\label{def:admissible}
	A function $\vp\in C^2(\Theta)$ is admissible at a point $(x_0,t_0)\in\Theta$ if either ${\nabla \vp(x_0,t_0)\not=0}$ or there are $\delta>0$, $f\in\mathcal{F}(F)$ and $\sigma\in\Sigma$ such that
	\begin{equation*}
		\abs{\vp(x,t)-\vp(x_0,t_0)-\partial_{t}\vp(x_0,t_0)(t-t_0)}\leq f(\abs{x-x_0})+\sigma(t-t_0),
	\end{equation*}
	for all $(x,t)\in B_\delta(x_0)\times (t_0-\delta, t_0+\delta)$.
\end{definition}
Note that by definition a function $\vp$ is automatically admissible at a point $(x_0,t_0)$ if either $\nabla\vp(x_0,t_0)\not=0$ or the function $-\vp$ is admissible at a point $(x_0,t_0)$.
\begin{definition}
	A function $u:\Theta\rightarrow\mathbb{R}\cup\left\{ \infty\right\} $
	is a viscosity supersolution to
	\[
	\partial_{t}u=\Delta_{p}^{q}u\quad\text{in }\Theta
	\]
	if the following three conditions hold.
	\begin{enumerate}
		\label{def:super}
		\item $u$ is lower semicontinuous,
		\item $u$ is finite in a dense subset of $\Theta$,
		\item whenever an admissible $\vp\in C^{2}(\Theta)$ touches $u$ at $\xi\in\Theta$
		from below, we have
		\[
		\begin{cases}
			\partial_{t}\vp(\xi)-\Delta_{p}^{q}\vp(\xi)\geq0 & \text{if }\nabla \vp(\xi)\not=0,\\
			\partial_{t}\vp(\xi)\geq0 & \text{if }\nabla \vp(\xi)=0.
		\end{cases}
		\]
	\end{enumerate}
	A function $u:\Theta\rightarrow\mathbb{R}\cup\left\{ -\infty\right\} $
	is a viscosity subsolution if $-u$ is a viscosity supersolution. A function $u:\Theta\rightarrow\mathbb{R} $
	is a viscosity solution if it is a supersolution and a subsolution.
\end{definition}
Note that if no admissible test function exists at a point $\xi$, the last condition is automatically satisfied.
If $q\geq2$, then viscosity solutions can be defined
in a standard way by using semicontinuous extensions, see Proposition 2.2.8 in \cite{Giga2006}.

%\begin{proof}
%	Assume thriving for a contradiction that there exists a $\xi_0\in\Theta$, such that
%	\begin{equation*}
%		v(x_0,t_0)-u(x_0,t_0)=\theta>0.
%	\end{equation*}
%	Define $v_\delta(x,t)=v(x,t)+C\delta\abs{\xi}^{\frac{q}{q-1}}$ where $C$ is a constant to be chosen later. We have
%	\begin{equation*}
%		v_\delta(x,t)-u(x,t)\leq C\delta\abs{\xi}^{\frac{q}{q-1}}
%	\end{equation*}
%	for all $(x,t)\in\partial\Theta$ and by choosing $\delta$ small enough, we have
%	\begin{equation*}
%		v_\delta(x,t)-u(x,t)\leq C\theta
%	\end{equation*}
%	for all $(x,t)\in\partial\Theta$. Define
%	\begin{equation*}
%		\Phi_\eps(x,t,y,s)=v_\delta(x,t)-u(y,s)-\Psi(x,t,y,s)
%	\end{equation*}
%	where $\Psi(x,t,y,s)=\frac{C}{\eps}\abs{x-y}^{\frac{q}{q-2}}+\frac{C}{\eps}\abs{t-s}^{2}$
%	and let $(x_\eps,t_\eps,y_\eps,s_\eps)$ be the point where $\Phi_\eps$ attains its maximum in $\overline{\Theta}\times\overline{\Theta}$. By increasing $\eps$ if necessary, we can guarantee that $(x_\eps,t_\eps,y_\eps,s_\eps)\in\Theta\times\Theta$. Using parabolic Theorem on sums \cite[Lemma 3.6.]{Koike2004}, we have that there exist
%	\begin{equation*}
%		(\partial_{t}\Psi,\nabla_x\Psi,X)\in\mathcal{\overline{P}}^{2,+}u(x_0,t_0) \text{ and } (-\partial_{t}\Psi,-\nabla_y\Psi,Y)\in\mathcal{\overline{P}}^{2,-}v(y_0,s_0)
%	\end{equation*}
%	with $X\leq Y$.
%	\begin{align*}
%		asd
%	\end{align*}
%\end{proof}
Our proofs use two different comparison principles. The first is the standard parabolic comparison principle, which is proven as Theorem 3.1 in \cite{Ohnuma1997}. Here we assume that the solutions are ordered on the parabolic boundary of the set.
\begin{theorem}
	\label{thm:comp}
	Suppose that u is a viscosity supersolution and $v$ is a viscosity subsolution to \eqref{eq:rgnppar} in $\Omega_{T}$. If
	$$
	\infty \neq \limsup _{\Omega_{T} \ni(y, s) \rightarrow(x, t)} v(y, s) \leq \liminf _{\Omega_{T} \ni(y, s) \rightarrow(x, t)} u(y, s) \neq-\infty
	$$
	for all $(x, t) \in \partial_{p} \Omega_{T},$ then $v \leq u$ in $\Omega_{T}$.
\end{theorem}
The second is the elliptic-type comparison principle which holds for arbitrary $\Theta$ as long as we compare over the entire Euclidean boundary. We state and prove this in the next section.

Consider the Dirichlet problem
\begin{equation}
	\label{eq:dirichlet}
	\begin{cases}
		\partial_t u=\abs{\nabla u}^{q-p}\div\left(\abs{\nabla u}^{p-2}\nabla u\right) & \text{ in } \Om_T \\
		u=g & \text{ on }\partial_{p}\Om_T.
	\end{cases}
\end{equation}
We have the following existence and uniqueness results for simple space-time cylinders.
\begin{theorem}
	\label{thm:existball}
	Let $g\in C(\partial_{p}B_T)$. Then there exists a unique viscosity solution $u\in C(\overbar{B}_T)$ to \eqref{eq:dirichlet} with $\Om_T=B_T$.
\end{theorem}
This theorem follows from \cite[Theorem 2.4.9]{Giga2006} and falls into the general framework studied by Ohnuma and Sato in \cite[Section 4]{Ohnuma1997}, where they prove the existence for the Cauchy problem in Corollary 4.10.

Imbert, Jin, and Silvestre proved the $C^{1,\alpha}$-regularity for solutions to \eqref{eq:rgnppar}. The time estimate we need follows from \cite[Lemma 3.1]{Imbert2019} and the space estimate follows from \cite[Corollary 2.4]{Imbert2019}. Combining these we get the following corollary. 

\begin{corollary}
	\label{cor:holderest}
	Let $u$ be a viscosity solution to \eqref{eq:rgnppar} in $Q_4$ and $\alpha\in(0,1)$. Then there exists a two positive constants $C_1$ and $C_2$ depending only on $n$, $p$, $q$ and $\norm{u}_{L^\infty(Q_4)}$ such that
	\begin{equation*}
		\abs{u(x,t)-u(y,s)}\leq C_1\abs{x-y}^\alpha+C_2\abs{t-s}^{\half}
	\end{equation*}
	for all $(x,t),(y,s)\in Q_1$.
\end{corollary}
Various regularity results for the non-homogeneous version of \eqref{eq:rgnppar} were proven by Attouchi in \cite{Attouchi2020} and Attouchi and Ruosteenoja in \cite{Attouchi2020b}.

In our proofs, we need the following stability result which is a special case of Theorem 5.2. in \cite{Imbert2019} which follows from Theorem 6.1 in \cite{Ohnuma1997}. We provide a proof by modifying the proof used for the $p$-parabolic equation, see \cite[Lemma 3.4]{Kilpelainen1996}.
\begin{lemma}
	\label{lem:bddseq}
	Suppose that $(u_i)_{i=1}^\infty$ is a locally uniformly bounded sequence of viscosity solutions to \eqref{eq:rgnppar} in $\Theta$. Then there exists a subsequence that converges locally uniformly in $\Theta$ to a viscosity solution $u$.
\end{lemma}
\begin{proof}
	The proof is based on a diagonalization argument.
	Let $(K_i)_{i=1}^\infty$ be a sequence of compact sets in $\Theta$ such that $K_i\subset K_{i+1}$ for all $i$ and
	\begin{equation*}
		\bigcup_{i=1}^\infty K_i=\Theta.
	\end{equation*} Let $\Xi_i=\{\xi_1,\xi_2,\dots\}$ be the set of points with rational coordinates in $K_i$ and define
	\begin{equation*}
		d_i=\frac{d(K_i,\partial\Theta)}{5}
	\end{equation*} for every $i$. For each $i$, define the family of sets
	\begin{equation*}
		U_i=\{\xi_j+Q_{d_i}\mid \xi_j\in\Xi_i\}=\{B_{d_i}(x_j)\times(t_j-d_i^2,0]\mid (x_j,t_j)\in\Xi_i\}.
	\end{equation*}The family $U_i$ is a countable cover of $K_i$ and by compactness and construction has a finite subcover $V_i$ formed over some finite index set $Z_i\subset\Xi_i$.
	
	Because each $u_k$ is a viscosity solution in each $\xi_j+Q_{d_i}$ and we chose $d_i$ to have enough space around the set, we can use Hölder continuity result Corollary \ref{cor:holderest} to get the estimate
	\begin{equation}
		\label{eq:bddseq1}
		\abs{u_k(x,t)-u_k(y,t)}\leq C\abs{x-y}^\alpha
	\end{equation}
	for any $(x,t),(y,t)\in \xi_j+Q_{d_i}$, $\alpha_{\xi_j}\in(0,1)$ and $k\in\N$, where $C=C(n, p,q, \norm{u_k}_{L^\infty(\xi_j+Q_{4d_i})})$. Because we assume that the sequence $(u_i)_{i=1}^\infty$ is locally uniformly bounded, we can pick a constants $C$ independent of $k$ and by taking maximum over all such $C$ and $\alpha$, we get that
	\begin{equation}
		\label{eq:bddseq2}
		\abs{u_k(x,t)-u_k(y,t)}\leq \hat{C}\abs{x-y}^{\hat{\alpha}}
	\end{equation} holds for every $k$ and some $\hat{C}$ and $\hat{\alpha}$. This estimate now holds for every ${(x,t),(y,t)\in\bigcup_{\xi_j\in Z_i}\left(\xi_j+Q_{d_i}\right)}$ so especially in $\overbar{K}_i$. Similarly by using Corollary \ref{cor:holderest} on each set and picking constants gets us the estimate\sloppy
	\begin{equation}
		\label{eq:bddseq3}
		\abs{u_k(x,t)-u_k(x,s)}\leq \hat{C}\abs{t-s}^{\half}
	\end{equation}for every $(x,t),(x,s)\in\bigcup_{\xi_j\in Z_i}\left(\xi_j+Q_{d_i}\right)$ so especially again in $\overbar{K}_i$. 
	
	Estimates \eqref{eq:bddseq2} and \eqref{eq:bddseq3} give us that the sequence $(u_k)_{k=1}^\infty$ is equicontinuous with respect to both space and time in $\overbar{K}_i$. Let $(u_k^i)_{k=1}^\infty$ be the subsequence given by Arzelà-Ascoli theorem that converges into a continuous function $u^i$ in $K_i$. Define a new sequence $(v_k)_{k=1}^\infty$ such that $v_k=u_k^k$ for all $k$. Now $v_k$ has a subsequence that converges locally uniformly in $\Theta$ to some continuous function $u$.
	
	Let us show that $u$ is a viscosity solution. Let $B_{t_1,t_2}\Subset \Theta$ for a ball $B$. Let $v$ be a viscosity solution in $B_{t_1,t_2}$, continuous on $\overbar{B}_{t_1,t_2}$ and taking the boundary values $v=u$ on $\partial_{p}B_{t_1,t_2}$. Such $v$ exists by Theorem \ref{thm:existball}. 
	
	By convergence for any $\eps>0$ 
	\begin{equation*}
		v-\eps=u-\eps<u_k<u+\eps=v+\eps
	\end{equation*}
	on $\partial_{p}B_{t_1,t_2}$ for large enough $k$. By comparison principle Theorem \ref{thm:comp}, we get
	\begin{equation*}
		v-\eps\leq u_k\leq v+\eps
	\end{equation*}
	in $B_{t_1,t_2}$ for each large $k$ and thus taking the limit as $k\to\infty$ gives us
	\begin{equation*}
		v-\eps\leq u\leq v+\eps
	\end{equation*}
	in $B_{t_1,t_2}$. Letting $\eps\to0$ gives us that $u$ is a viscosity solution in $B_{t_1,t_2}$ because the solution $v$ given by Theorem \ref{thm:existball} is unique.
\end{proof}
Later we prove that Perron solutions are actually viscosity solutions and for this proof, we need the concept of parabolic modification. 
\begin{definition}
	\label{def:parmod}
	Let $B_{t_1,t_2}\Subset\Theta$ and $u$ be a viscosity supersolution to \eqref{eq:rgnppar} in $\Theta$ and bounded in $\Om_T$. We define the parabolic modification of $u$ in $B_{t_1,t_2}$ as
	\begin{equation*}
		U=\begin{cases}
			v &\text{ in }B_{t_1,t_2},\\
			u &\text{ in }\Theta\setminus B_{t_1,t_2},
		\end{cases}
	\end{equation*}
	where
	\begin{equation*}
		v(\xi)=\sup\{h(\xi)\mid h\in C(\overline{B}_{t_1,t_2}) \text{ is a viscosity solution to \eqref{eq:rgnppar} and } h\leq u \text{ on } \partial_{p} B_{t_1,t_2} \}.
	\end{equation*}
\end{definition}
Clearly $U\leq u$ in $\Theta$ because by comparison principle Theorem \ref{thm:comp}, each $h\leq u$ in $B_{t_1,t_2}$ and thus also $v\leq u$ in $B_{t_1,t_2}$.
\begin{lemma}
	\label{le:parmod}
	Let $B_{t_1,t_2}\Subset\Theta$ and $u$ be a viscosity supersolution to \eqref{eq:rgnppar} in $\Theta$ and bounded in $B_{t_1,t_2}$. Then the parabolic modification $U$ is a viscosity supersolution in $\Theta$ and a viscosity solution in $B_{t_1,t_2}$. 
\end{lemma}
\begin{proof}
	Let $(\theta_i)_{i=1}^\infty$ be an increasing sequence of continuous functions on $\partial_{p}B_{t_1,t_2}$ such that
	\begin{equation*}
		u=\lim_{i\to\infty}\theta_i
	\end{equation*}
	on $\partial_{p}B_{t_1,t_2}$. By Theorem \ref{thm:existball}, there exists a sequence $(h_i)_{i=1}^\infty$ of viscosity solutions on $B_{t_1,t_2}$ such that $h_i$ coincides with $\theta_i$ on $\partial_{p}B_{t_1,t_2}$. Using the comparison principle Theorem \ref{thm:comp} pairwise for each $h_i$, we get that the sequence $(h_i)_{i=1}^\infty$ is increasing on $\overline{B}_{t_1,t_2}$ and that the limit function is $v$ in the definition of the parabolic modification. Moreover, since the sequence is bounded, the limit function $v$ is also a viscosity solution to \eqref{eq:rgnppar} in $B_{t_1,t_2}$ by Lemma \ref{lem:bddseq}. The function $U$ is a viscosity supersolution in $\Theta\setminus\overline{B}_{t_1,t_2}$ by definition, so only the boundary of two sets is left.
	
	Let $\vp$ be an admissible test function touching $U$ from below at $\xi\in\partial B_{t_1,t_2}$. We have $\vp(\xi)=U(\xi)$ and $\vp<U$ in some neighborhood $V$ of $\xi$. This $\vp$ is also an admissible test function touching $u$ from below at $\xi$ because $\vp(\xi)=U(\xi)=u(\xi)$ and $\vp(\zeta)<U(\zeta)\leq u(\zeta)$ for all $\zeta\in V$. Because $u$ is a viscosity supersolution in the entire $\Theta$, we necessarily have
	\[
	\begin{cases}
		\partial_{t}\vp(\xi)-\Delta_{p}^{q}\vp(\xi)\geq0 & \text{if }\nabla \vp(\xi)\not=0,\\
		\partial_{t}\vp(\xi)\geq0 & \text{if }\nabla \vp(\xi)=0.
	\end{cases}
	\]
	which now implies that condition (3) of Definition \ref{def:super} holds for $U$ at $\xi$. Thus $U$ is a viscosity solution in $B_{t_1,t_2}$ and a viscosity supersolution in the entire $\Theta$.
\end{proof}
%\begin{lemma}
%	Let $v_1$ and $v_2$ be viscosity supersolutions to equation \eqref{eq:rgnppar} in $\Theta$. Then
%	\begin{equation*}
%		u=\min\{v_1,v_2\}
%	\end{equation*}
%	is a viscosity supersolution to  equation \eqref{eq:rgnppar} in $\Theta$.
%\end{lemma}
%\begin{proof}
%	This is clear for all $\xi\in\Theta$ such that $u\equiv v_i$ in some neighborhood of $\xi$ for $i\in\{1,2\}$. We need to check the points $\xi$ where $v_1(\xi)=v_2(\xi)$.
%	
%	Let $\vp$ be an admissible test function touching $u$ from below at $\xi$. We have $\vp(\xi)=u(\xi)=v_1(\xi)=v_2(\xi)$ and $\vp(\zeta)<u(\zeta)$ in some neighborhood $V$ of $\xi$. But this implies that
%	\begin{equation*}
%		\vp(\zeta)<v_1(\zeta) \text{ and } \vp(\zeta)<v_2(\zeta)
%	\end{equation*}
%	for all $\zeta\in V$. Test function being admissible does not depend on the function we are touching and thus $\vp$ is also an admissible test function touching $v_1$ and $v_2$ from below at $\xi$. Thus because these are viscosity supersolutions, we necessarily have
%	\[
%	\begin{cases}
%		\partial_{t}\vp(x,t)-\Delta_{p}^{q}\vp(x,t)\geq0 & \text{if }\nabla \vp(x,t)\not=0,\\
%		\partial_{t}\vp(x,t)\geq0 & \text{if }\nabla \vp(x,t)=0.
%	\end{cases}
%	\]
%	which now implies that $u$ is a viscosity supersolution.
%\end{proof}
%In the later sections of the paper we work with radial functions. We say a function $u:\Theta\to\R$ is radial if there exists a function $v:\R\times\R^+\to\R$ such that $u(x,t)=v(\abs{x},t)$ for all $(x,t)\in\Theta$.
\section{Elliptic-type comparison principle}
\label{sec:ecomp}
The comparison principle for general bounded open sets $\Theta$ has not been proven for equation \eqref{eq:rgnppar} before and we need it for a few proofs in this paper. Hence we will provide a proof using the doubling of variables method and the Theorem on sums which is the standard strategy often used to prove comparison principle for viscosity solutions of equations of this type. 

\begin{theorem}
	\label{thm:ecomp}
	Suppose that u is a viscosity supersolution and $v$ is a viscosity subsolution to \eqref{eq:rgnppar} in $\Theta$.  If
	$$
	\infty \neq \limsup _{\Theta \ni(y, s) \rightarrow(x, t)} v(y, s) \leq \liminf _{\Theta \ni(y, s) \rightarrow(x, t)} u(y, s) \neq-\infty
	$$
	for all $(x, t) \in \partial\Theta,$ then $v \leq u$ in $\Theta$.
\end{theorem}
Before the proof, we will define notation and prove some lemmas we need.

\begin{lemma}
	\label{le:admi}
	Assume $\vp\in C^2(\Theta)$ is an admissible test function at $(x_0,t_0)\in\Theta$ and let $T=\sup\{t\in\R\mid (x,t)\in\Theta\}$. If $\nabla\vp(x_0,t_0)=0$, then
	\begin{equation*}
		\psi(x,t)=\vp(x,t)+\frac{\gamma}{T-t}
	\end{equation*}
	is also admissible at $(x_0,t_0)$ for all $\gamma>0$.
\end{lemma}
\begin{proof}
	We have $\partial_t\psi(x,t)=\partial_{t}\vp(x,t)+\frac{\gamma}{(T-t)^2}$. Because we assumed that $\vp$ is admissible at $(x_0,t_0)$, there exists a $\delta>0$, $f\in\mathcal{F}(F)$ and $\sigma_1\in\Sigma$ such that
	\begin{align*}
		\label{eq:admi1}
		&\abs{\psi(x,t)-\psi(x_0,t_0)-\partial_t\psi(x_0,t_0)(t-t_0)}\\&=\abs{\vp(x,t)+\frac{\gamma}{T-t}-\vp(x_0,t_0)-\frac{\gamma}{T-t_0}-\left(\partial_{t}\vp(x_0,t_0)+\frac{\gamma}{(T-t_0)^2}\right)(t-t_0)} \\
		&\leq f(\abs{x-x_0})+\sigma_1(t-t_0)+\abs{\frac{\gamma}{T-t}-\frac{\gamma}{T-t_0}-\frac{\gamma (t-t_0)}{(T-t_0)^2}}\\
		&\leq f(\abs{x-x_0})+\sigma_1(t-t_0)+\abs{h(t)-h(t_0)-h'(t_0)(t-t_0)},\numberthis
	\end{align*}
	for all $(x,t)\in B_{\delta}(x_0)\times (t_0-\delta, t_0+\delta)\subset\Theta$. Here
	\begin{equation*}
		h(t)=\frac{\gamma}{T-t}, 
	\end{equation*}which is smooth for $t\in(t_0-\delta, t_0+\delta)$ as this interval does not contain $T$. By Taylor's theorem using the Lagrange form for the remainder, there exists $c\in(t,t_0)$ such that
	\begin{equation*}
		h(t)=h(t_0)+h'(t_0)(t-t_0)+\frac{h''(c)}{2}(t-t_0)^2.
	\end{equation*}
	Because $h''(x)$ is bounded in $(t_0-\delta, t_0+\delta)$, we can estimate the last term of \eqref{eq:admi1} by
	\begin{equation*}
		\abs{h(t)-h(t_0)-h'(t_0)(t-t_0)}=\abs{\frac{h''(c)}{2}(t-t_0)^2}\leq \sup_{c\in(t_0-\delta, t_0+\delta)}h''(c)(t-t_0)^2=\sigma_2(t-t_0).
	\end{equation*}
	This $\sigma_2$ is even, satisfies $\sigma_2(0)=\sigma_2'(0)=0$ and $\sigma_2(r)>0$ for all $r>0$ and thus $\sigma_2\in\Sigma$. Combining this with estimate \eqref{eq:admi1}, we have
	\begin{equation}
		\abs{\psi(x,t)-\psi(x_0,t_0)-\partial_t\psi(x_0,t_0)(t-t_0)}\leq f(\abs{x-x_0})+\sigma(t-t_0),
	\end{equation}
	for all $(x,t)\in B_\delta(x_0)\times (t_0-\delta, t_0+\delta)$. Here $f\in\mathcal{F}(F)$ and $\sigma=\sigma_1+\sigma_2\in\Sigma$ and thus $\psi$ is admissible at point $(x_0,t_0)$.
\end{proof}
Next we will define some notation used in the proof of the comparison principle. Let \sloppy ${T=\sup\{t\in\R\mid (x,t)\in\Theta\}}$, $\eps>0$, $\gamma>0$ and $f\in\mathcal{F}(F)$ and define
\begin{equation}
	\label{eq:ecompphi}
	\Phi(x,t,y,s)=v(x,t)-u(y,s)-\Psi(x,t,y,s)
\end{equation}
where
\begin{equation}
	\label{eq:ecomppsi}
	\Psi(x,t,y,s)=\frac{1}{\eps}f(\abs{x-y})+\frac{1}{\eps}(t-s)^{2}+\frac{\gamma}{T-s}+\frac{\gamma}{T-t}.
\end{equation}
By our assumptions for $u$ and $v$, the function $v(x,t)-u(y,s)$ is upper semicontinuous and bounded from above by some constant $M$ in $\overline{\Theta}\times\overline{\Theta}$. Thus it attains its maximum in this set and by continuity of $\Psi$, so does the function $\Phi$.
\begin{lemma}
	\label{le:inte}
	Let $\xi_\eps=(x_\eps,t_\eps,y_\eps,s_\eps)$ be the point where $\Phi$ attains its maximum in $\overline{\Theta}\times\overline{\Theta}$ and assume $v(x_\eps,t_\eps)-u(x_\eps,t_\eps)=\theta>0$. Then there exists a constant $\gamma_0>0$, such that
	\begin{equation}
		\lim_{\eps\to0}\abs{x_\eps-y_\eps}=0 \quad \text{ and }\quad  \lim_{\eps\to0}\abs{t_\eps-s_\eps}=0
	\end{equation} for all $\gamma<\gamma_0$. There also exists a constant $\eps_0=\eps_0(\gamma_0)>0$, such that $\xi_\eps\in\Theta\times\Theta$ for all $\eps<\eps_0$.
\end{lemma}
\begin{proof}
	We will first show a lower bound for $\Phi(x_\eps,t_\eps,y_\eps,s_\eps)$. If we choose $\gamma_0$ to satisfy
	\begin{equation*}
		\frac{2\gamma_0}{T-t_0}\leq\frac{\theta}{2},
	\end{equation*}
	we have that
	\begin{equation}
		\label{eq:ecomp1}
		\Phi(x_\eps,t_\eps,y_\eps,s_\eps)\geq \Phi(x_\eps,t_\eps,x_\eps,t_\eps)=\theta-\frac{2\gamma}{T-t_0}\geq\frac{\theta}{2},
	\end{equation}
	for all $\gamma<\gamma_0$.
	 By equation \eqref{eq:ecomp1}, we also have
	\begin{equation}
		\label{eq:ecomp2}
		\Psi(x_\eps,t_\eps,y_\eps,s_\eps)< v(x_\eps,t_\eps)-u(y_\eps,s_\eps).
	\end{equation}
	Because we took $f\in \mathcal{F}(F)$, we know it is necessarily monotone increasing in $\R^+$ by definition. Thus there exists a monotone increasing inverse function $f^{-1}:\R^+\to\R^+$. Using this and the fact that
	\begin{equation*}
		f(\abs{x-y})\leq\eps\Psi(x,t,y,s),
	\end{equation*} we can conclude by inequality \eqref{eq:ecomp2}
	\begin{equation*}
		\abs{x_\eps-y_\eps}=f^{-1}\left(f(\abs{x_\eps-y_\eps})\right)\leq f^{-1}\left(\eps\Psi(x_\eps,t_\eps,y_\eps,s_\eps)\right)\leq f^{-1}\left(\eps M\right)
	\end{equation*}
	and
	\begin{equation*}
		\abs{t_\eps-s_\eps}\leq\left(\eps\Psi(x_\eps,t_\eps,y_\eps,s_\eps)\right)^{\frac{1}{2}}\leq \left(\eps M\right)^{\frac{1}{2}}.
	\end{equation*}
	Taking limits as $\eps\to0$ on both sides, these together imply
	\begin{equation}
		\label{eq:ecomp3}
		\lim_{\eps\to0}\abs{x_\eps-y_\eps}=0 \quad \text{ and }\quad  \lim_{\eps\to0}\abs{t_\eps-s_\eps}=0
	\end{equation} for all $\gamma<\gamma_0$.
	
	We have all the tools we need to show that $\xi_0\in\Theta\times\Theta$. Thriving for a contradiction, assume that $\eps_0$ stated in the theorem does not exist. Then necessarily there exists sequences $(\eps_i)_{i=1}^\infty$ and $(\gamma_i)_{i=1}^\infty\subset(0,\gamma_0)$, such that $\lim_{i\to\infty}\eps_i=0$ and $\Phi$ defined with $\eps=\eps_i$ and $\gamma=\gamma_i$ attains its maximum at a point $(x_i,t_i,y_i,s_i)\in\partial(\Theta\times\Theta)$. Because $\partial(\Theta\times\Theta)$ is compact and \eqref{eq:ecomp3} holds, there exists a point $(\hat{x},\hat{t})\in\partial\Theta$, such that
	\begin{equation*}
		\lim_{i\to\infty}x_i=\lim_{i\to\infty}y_i=\hat{x} \quad \text{ and }\quad \lim_{i\to\infty}t_i=\lim_{i\to\infty}s_i=\hat{t}
	\end{equation*}
	passing to a subsequence if necessary. But now estimate \eqref{eq:ecomp2} implies
	\begin{align*}
		0&<\limsup_{i\to \infty}\Psi(x_i,t_i,y_i,s_i)\\&\leq\limsup_{i\to \infty}(v(x_i,t_i)-u(y_i,s_i))\\&\leq\limsup_{i\to \infty}v(x_i,t_i)-\liminf_{i\to \infty}u(y_i,s_i)\\&\leq\limsup _{\Theta \ni(y, s) \rightarrow(\hat{x}, \hat{t})} v(y, s) - \liminf _{\Theta \ni(y, s) \rightarrow(\hat{x}, \hat{t})} u(y, s)\leq0,
	\end{align*}
	where we used our assumption about the functions on the boundary. This contradiction proves that there exists a $\eps_0>0$, such that $\xi_0\in\Theta\times\Theta$ for all $\eps<\eps_0$.
\end{proof}

Now we are ready to prove the elliptic-type comparison principle.
\begin{proof}[Proof of Theorem \ref{thm:ecomp}]
	Assume thriving for a contradiction that there exists a $(x_\eps,t_\eps)\in\Theta$, such that
	\begin{equation*}
		v(x_\eps,t_\eps)-u(x_\eps,t_\eps)=\theta>0.
	\end{equation*}
	Let $\Phi$ and $\Psi$ be defined as before in \eqref{eq:ecompphi} and \eqref{eq:ecomppsi} and let $\xi_\eps=(x_\eps,t_\eps,y_\eps,s_\eps)$ be the point where $\Phi$ attains its maximum in $\overline{\Theta}\times\overline{\Theta}$. Note that this point depends on $\eps$ and $\delta$. By Lemma \ref{le:inte}, there exists constants $\gamma_0$ and $\eps_0$ such that $\xi_\eps\in\Theta\times\Theta$ for all $\eps<\eps_0$ and
		\begin{equation}
		\lim_{\eps\to0}\abs{x_\eps-y_\eps}=0 \quad \text{ and }\quad  \lim_{\eps\to0}\abs{t_\eps-s_\eps}=0
	\end{equation} for all $\gamma<\gamma_0$.
	Let
	\begin{equation*}
		\vp^+(x,t)=\frac{1}{\eps}f(\abs{x-y_\eps})+\frac{1}{\eps}(t-s_\eps)^{2}+\frac{\gamma}{T-t},
	\end{equation*}
	and
	\begin{equation*}
	\vp^-(y,s)=-\frac{1}{\eps}f(\abs{x_\eps-y})-\frac{1}{\eps}(t_\eps-s)^{2}-\frac{\gamma}{T-s}.
	\end{equation*}
	
	For every $\eps<\eps_0$ and $\gamma<\gamma_0$, there are two possible cases. First if $x_\eps=y_\eps$, we have $\nabla\vp^+(x_\eps,t_\eps)=\nabla\vp^-(y_\eps,s_\eps)=0$. 
	These are admissible test functions at points $(x_\eps,t_\eps)$ and $(y_\eps,s_\eps)$ respectively by Lemma \ref{le:admi}. The function $\vp^+$, adding a constant if necessary, touches $v$ from above at $(x_\eps,t_\eps)$, and hence by the definition of a viscosity subsolution, we have
	\begin{equation*}
		\partial_t\vp^+(x_\eps,t_\eps)=\frac{2}{\eps}(t_\eps-s_\eps)+\frac{\gamma}{(T-t_\eps)^2}\leq 0.
	\end{equation*}
	Similarly, $\vp^-$ with a possible added constant touches $u$ from below at $(y_\eps,s_\eps)$, and hence by the definition of a viscosity supersolution, we have
	\begin{equation*}
		\partial_t\vp^-(y_\eps,s_\eps)=\frac{2}{\eps}(t_\eps-s_\eps)-\frac{\gamma}{(T-s_\eps)^2}\geq 0.
	\end{equation*}
	Hence
	\begin{align*}
		0\leq\partial_t\vp^-(y_\eps,s_\eps)-\partial_t\vp^+(x_\eps,t_\eps)&=-\frac{\gamma}{(T-s_\eps)^2}-\frac{\gamma}{(T-t_\eps)^2}<0,
	\end{align*}
	which is a contradiction.

	In the second case, we have $x_\eps\not=y_\eps$. For such $\eps$ and $\gamma$, we have $\nabla\vp^+(x_\eps,t_\eps)\not=0$ and $\nabla\vp^-(y_\eps,s_\eps)\not=0$. We denote parabolic superjet by $\mathcal{P}^{2,+}$ and subjet by $\mathcal{P}^{2,-}$ and their closures by $\mathcal{\overline{P}}^{2,-}$ and $\mathcal{\overline{P}}^{2,-}$ respectively. For definitions of these, we direct the reader to see \cite{Crandall1992}. We can use elliptic Theorem on sums in dimension $n+1$, see for example \cite[Lemma 3.6]{Koike2004}, with \cite[Lemma 3.5]{Ohnuma1997}. By this, we conclude that there exist matrixes $X,Y\in S^n$, such that
	\begin{equation*}
		(\partial_{t}\Psi(\xi_\eps),\nabla_x\Psi(\xi_\eps),X)\in\mathcal{\overline{P}}^{2,+}v(x_\eps,t_\eps) \text{ and } (-\partial_{s}\Psi(\xi_\eps),-\nabla_y\Psi(\xi_\eps),Y)\in\mathcal{\overline{P}}^{2,-}u(y_\eps,s_\eps),
	\end{equation*}
	and $X\leq Y$. Notice that because we assumed $x_\eps\not= y_\eps$, we have
	\begin{equation}
		\label{eq:ecomp4}
		\nabla_x\Psi(\xi_0)=-\nabla_y\Psi(\xi_0)\not=0
	\end{equation} and $F(\eta,X)$ is continuous in some neighborhood and we do not have to worry about the admissibility of test functions. Thus the viscosity solutions can be equivalently defined using semijets in this neighborhood, see \cite{Crandall1992}.  Since $u$ is a viscosity supersolution, this definition with \eqref{eq:ecomp4} implies
	\begin{align*}
		\label{eq:ecomp5}
		0&\leq-\partial_s\Psi(\xi_\eps)-F(-\nabla_y\Psi(\xi_\eps),Y)=-\left(-\frac{2}{\eps}(t_\eps-s_\eps)+\frac{\gamma}{(T-s_\eps)^2}\right)-F(\nabla_x\Psi(\xi_\eps),Y)
		\\&=\frac{2}{\eps}(t_\eps-s_\eps)-\frac{\gamma}{(T-s_\eps)^2}-F(\nabla_x\Psi(\xi_\eps),Y)\numberthis
	\end{align*}
	and because $v$ is a subsolution
	\begin{align}
		\label{eq:ecomp6}
		0\geq\partial_t\Psi(\xi_\eps)-F(\nabla_x\Psi(\xi_\eps),X)=\frac{2}{\eps}(t_\eps-s_\eps)+\frac{\gamma}{(T-t_\eps)^2}-F(\nabla_x\Psi(\xi_\eps),X).
	\end{align}
	Subtracting \eqref{eq:ecomp5} from \eqref{eq:ecomp6}, we get using degenerate ellipticity of $F$
	\begin{align*}
		0&\geq \frac{\gamma}{(T-t_\eps)^2}+\frac{\gamma}{(T-s_\eps)^2}+F(\nabla_x\Psi(\xi_\eps),Y)-F(\nabla_x\Psi(\xi_\eps),X)
		\\&\geq \frac{\gamma}{(T-t_\eps)^2}+\frac{\gamma}{(T-s_\eps)^2}
		\\&>0.
	\end{align*}
	Both cases lead to a contradiction and this concludes the proof.
\end{proof}
\section{Perron solutions}
\label{sec:perron}
One common way of solving the Dirichlet problem in arbitrary domains is the Perron method. For our uses, it is enough to consider bounded domains with bounded boundary data. The idea is to construct an upper solution to the Dirichlet problem as a point-wise infimum over a suitable class of supersolutions. We prove that for bounded boundary data, this construction gives us a viscosity solution inside the set. We use the notation from \cite{Bjorn2015} for Perron solutions.
\begin{definition}
	Let $f:\partial\Theta:\to\R$ be a bounded function. The upper class $\mathcal{U}_f$ is defined to be the class of all viscosity supersolutions $u$ to equation \eqref{eq:rgnppar} in $\Theta$ which are bounded from below and such that
	\begin{equation}
		\liminf_{\Theta\ni\eta\to\xi}u(\eta)\geq f(\xi) \quad \text{for all } \xi\in\partial\Theta.
	\end{equation}
	The upper Perron solution of $f$ is defined as
	\begin{equation*}
		\Hu f(\xi)=\inf_{u\in\mathcal{U}_f}u(\xi), \quad \xi\in\Theta.
	\end{equation*}
	Similarly, the lower class $\mathcal{L}_f$ is defined to be the class of all viscosity subsolutions $u$ to equation \eqref{eq:rgnppar} in $\Theta$ which are bounded from above and such that
	\begin{equation}
		\limsup_{\Theta\ni\eta\to\xi}u(\eta)\leq f(\xi) \quad \text{for all } \xi\in\partial\Theta,
	\end{equation}
	and define the lower Perron solution of $f$ by
	\begin{equation*}
		\Hl f(\xi)=\sup_{u\in\mathcal{L}_f}u(\xi),\quad \xi\in\Theta.
	\end{equation*}
\end{definition}
Note that $\mathcal{L}_f$ and $\mathcal{U}_f$ are always non-empty for bounded $f$. We in fact have $u\in\mathcal{U}_f$ for any constant function larger than $\sup_{\xi\in\partial\Theta}f(\xi)$ and similar result for the lower class. It is also clear that for bounded $f$, the definition of Perron solutions does not change if we restrict $\mathcal{L}_f$ and $\mathcal{U}_f$ only to bounded functions.

In the next theorem, we prove that Perron solutions are in fact viscosity solutions. This result is quite classical and a similar proof works for different equations as it only uses the stability result, parabolic modification, and basic properties of Perron solutions. We provide a proof for the convenience of the reader.
\begin{theorem}
	\label{thm:perronarevisc}
	If the boundary function $f:\partial\Theta\to\R$ is bounded, then Perron solutions $\Hu f$ and $\Hl f$ are viscosity solutions to \eqref{eq:rgnppar} in $\Theta$.
\end{theorem}
\begin{proof}
	We mainly follow the argument in Kilpeläinen and Lindqvist in \cite[Theorem 5.1]{Kilpelainen1996}.
	Fix a space-time cylinder $B_{t_1,t_2}\Subset\Theta$. Choose a countable, dense subset
	\begin{equation*}
		\Xi=\{\xi_1,\xi_2,\dots\}
	\end{equation*}
	of $B_{t_1,t_2}$. For each $j=1,2,\dots$, we choose a sequence of functions $u_{i,j}\in\mathcal{U}_f$ with $i=1,2,\dots$, such that
	\begin{equation*}
		\lim_{i\to\infty} u_{i,j}(\xi_j)=\Hu f(\xi_j).
	\end{equation*}
	We may assume that each $u_{i,j}$ is bounded. Now define
	\begin{equation*}
		v_{i,j}(\xi)=\min_{1\leq m\leq j}\{u_{i,m}(\xi)\}
	\end{equation*}
	for each $j$ and $i$. The minimum of two viscosity supersolutions is also a viscosity supersolution by standard arguments and by iterating this, we get that each $v_{i,j}$ is a viscosity supersolution to $\eqref{eq:rgnppar}$ in $\Theta$ and $v_{i,j}\in\mathcal{U}_f$.
	By definition $v_{i,j}(\xi)\geq\Hu f(\xi)$ for all $i$ and $j$ and now by construction $v_{i,k}(\xi)\geq v_{i,j}(\xi)$ for each $k=1,2,\dots,j$.
	Thus for these indexes, we have
	\begin{equation*}
		\Hu f(\xi_k)\leq  v_{i,j}(\xi_k)\leq  v_{i,k}(\xi_k)
	\end{equation*} and taking limits as $i\to\infty$, we get that for any $j$, the sequence we now have satisfies
	\begin{equation}
	\label{eq:perronarevisc1}
		\lim_{i\to\infty} v_{i,j}(\xi_k)=\Hu f(\xi_k)
	\end{equation}
	for each $k=1,2,\dots,j$.
	
	Let $V_{i,j}$ be the parabolic modification of $v_{i,j}$ in $B_{t_1,t_2}$ according to Definition \ref{def:parmod}. Now 
	\begin{equation*}
		\Hu f\leq V_{i,j}\leq v_{i,j}
	\end{equation*}
	by definition and $V_{i,j}$ is a viscosity solution in $B_{t_1,t_2}$.
	
	By passing to a subsequence if necessary, we get from Lemma \ref{lem:bddseq} that for any $j$, the sequence $(V_{i,j})_{i=1}^\infty$ converges locally uniformly to a viscosity solution $v_j$ in $B_{t_1,t_2}$. Again by Lemma \ref{lem:bddseq}, the sequence $(v_j)_{j=1}^\infty$ has a subsequence that converges locally uniformly to a viscosity solution $h$ in $B_{t_1,t_2}$. By the construction, it holds
	\begin{equation*}
		h\geq\Hu f
	\end{equation*}
	in $B_{t_1,t_2}$ and by equation \eqref{eq:perronarevisc1}, the equality $h=\Hu f$ holds in the dense subset $\Xi\subset B_{t_1,t_2}$. Take $v\in\mathcal{U}_f$ and let $V$ be its parabolic modification in $B_{t_1,t_2}$. By definition $v\geq V$ and $V\geq \Hu f$ in $B_{t_1,t_2}$. Also because $h=\Hu f$ in a dense subset, we have by continuity of $V$ and $h$ $v\geq V\geq h$ in $B_{t_1,t_2}$ and thus taking infimum over all $v\in\mathcal{U}_f$, we get
	\begin{equation*}
		h\leq\Hu f
	\end{equation*}
	in $B_{t_1,t_2}$. It follows that $\Hu f=h$, so it is a viscosity solution for any cylinder $B_{t_1,t_2}$ and thus in $\Theta$. The lower Perron solution $\Hl f$ is treated analogously.
\end{proof}
To finish this section, we will prove the so-called pasting lemma that plays a key role in our proofs. It is kind of similar to the parabolic modification we used before but defined for arbitrary open sets. This is a useful tool when constructing suitable new viscosity supersolutions to be used as barriers.
\begin{lemma}{(Pasting lemma)}
	\label{le:pasting}
	Let $G\subset\Theta$ be open. Also let $u$ and $v$ be viscosity supersolutions to \eqref{eq:rgnppar} in $\Theta$ and $G$ respectively, and let
	\begin{equation*}
		w=\begin{cases}
			\min\{u,v\} & \text{ in }G,\\
			u			& \text{ in }\Theta\setminus G.
		\end{cases}
	\end{equation*}
	If $w$ is lower semicontinuous, then $w$ is a viscosity supersolution to \eqref{eq:rgnppar} in $\Theta$.
	\begin{proof}
		By assumption, $w$ is lower semicontinuous and by construction $w$ is finite in a dense subset of $\Theta$. Because the minimum of two viscosity supersolutions is a viscosity supersolution by standard arguments, we only need to verify the condition (3) of the Definition \ref{def:super} for $\xi\in\partial G$.
		
		Let $\vp$ be an admissible test function touching $w$ from below at $\xi\in\partial G$. We have $\vp(\xi)=w(\xi)$ and $\vp<w$ in some neighborhood $V$ of $\xi$. This $\vp$ is also an admissible test function touching $u$ from below at $\xi$ because $\vp(\xi)=w(\xi)=u(\xi)$ and $\vp(\zeta)<w(\zeta)\leq u(\zeta)$ for all $\zeta\in V$. Because $u$ is a viscosity supersolution, we necessarily have
		\[
		\begin{cases}
			\partial_{t}\vp(\xi)-\Delta_{p}^{q}\vp(\xi)\geq0 & \text{if }\nabla \vp(\xi)\not=0,\\
			\partial_{t}\vp(\xi)\geq0 & \text{if }\nabla \vp(\xi)=0.
		\end{cases}
		\]
		which now implies that condition (3) holds for $w$ at $\xi$. Thus $w$ is a viscosity supersolution in the entire $\Theta$.
	\end{proof}
\end{lemma}
\section{Barriers and boundary regularity}
\label{sec:barrier}
In this section, we define regular boundary points and barrier functions and prove how barriers can be used to characterize boundary regularity. It turns out that a boundary point is regular if and only if there exists a family of barrier functions. We show by a counterexample that a single barrier does not suffice when $q<2$ and this remains an open problem for $q>2$. This problem is open even for the usual $p$-parabolic equation where the barrier approach has been examined in \cite{Bjorn2015} and \cite{Kilpelainen1996}. We start with definitions.
\begin{definition}
	\label{def:reg}
	A boundary point $\xi_0\in\partial\Theta$ is \textit{regular} to equation \eqref{eq:rgnppar} if
	\begin{equation*}
		\liminf_{\Theta\ni\xi\to\xi_0}\Hu f(\xi)=f(\xi_0)
	\end{equation*}
	for every $f:C(\partial\Theta)\to\R$. If the set is ambiguous from the context, we will specify that a point is regular with respect to the set $\Theta$. %Definition for the equation \eqref{eq:radial eq} is the same but we take limit of $\Hud f(\xi)$.
\end{definition}
Since $\Hl f=-\Hu (-f)$, regularity can be equivalently defined using lower Perron solutions.

Next, we will define barriers and barrier families.
\begin{definition}
	Let $\xi_0\in\partial\Theta$. A function $w:\Theta\to(0,\infty]$ is a barrier to \eqref{eq:rgnppar} in $\Theta$ at point $\xi_0$ if
	\begin{itemize}
		\item[(a)] $w$ is a positive viscosity supersolution to equation \eqref{eq:rgnppar} in $\Theta$,
		\item[(b)] $\liminf_{\Theta\ni\zeta\to\xi_0}w(\zeta)=0$,
		\item[(c)]$\liminf_{\Theta\ni\zeta\to\xi}w(\zeta)>0 \text{ for }\xi\in\partial\Theta\setminus\{\xi_0\}$.
	\end{itemize}
	
\end{definition}
	We define barriers to be viscosity supersolutions which is not the standard definition in the recent literature, where barriers are often defined through the comparison principle. These two definitions are equivalent as we will prove in Lemma \ref{le:barrierdef}.
\begin{definition}
	\label{def:barfam}
	Let $\xi_0\in\partial\Theta$. A family of functions $w_j:\Theta\to(0,\infty]$, $j=1,2,\dots$, is a \textit{barrier family} to \eqref{eq:rgnppar} in $\Theta$ at point $\xi_0$ if for each $j$,
	\begin{itemize}
		\item[(a)] $w_j$ is positive viscosity supersolution to equation \eqref{eq:rgnppar} in $\Theta$,
		\item[(b)] $\liminf_{\Theta\ni\zeta\to\xi_0}w_j(\zeta)=0$,
		\item[(c)] for each $k=1,2,\dots$, there is a $j$ such that
		\begin{equation*}
			\liminf_{\Theta\ni\zeta\to\xi} w_j(\zeta)\geq k \quad \text{ for all }\xi\in\partial\Theta \text{ with } \abs{\xi-\xi_0}\geq \frac{1}{k}.
		\end{equation*}
	\end{itemize}
	We say that the family $w_j$ is a \textit{strong barrier family} in $\Theta$ at the point $\xi_0$ if,
	in addition, the following conditions hold:
	\begin{itemize}
		\item[(d)]$w_j$ is continuous in $\Theta$,
		\item[(e)]there is a non-negative function $d\in C(\overbar{\Theta})$, with $d(z)=0$ if and only if $z=\xi_0$ such that for each $k=1,2,\dots,$ there is a $j=j(k)$ such that $w_j\geq kd$ in $\Theta$.
	\end{itemize}
\end{definition}
%\begin{remark}
%	Barriers and barrier families are defined similarly for equation \eqref{eq:radial eq}. We just require $w$ to be superparabolic function to equation \eqref{eq:radial eq} instead in condition (a).
%\end{remark}

In the following lemma, we will evaluate the operators for a prototype barrier function. Most barriers we use consist of sums of these functions and thus these formulas make our calculations in the proofs easier.
\begin{lemma}
	\label{le:calc}
	Let $C>0$, $\beta\in\R$ and
	\begin{equation*}
		v(x,t)=C\abs{x}^\frac{q}{q-1} t^\beta.
	\end{equation*}
	Then for all $t>0$, we have
	\begin{equation*}
	\Delta_{p}^{q}v(x,t)=\left(C\frac{q}{q-1}\right)^{q-1}\left(n+\frac{p-q}{q-1}\right)t^{\beta(q-1)}.
	\end{equation*}
	and
	\begin{equation*}
		\partial_{t}v(x,t)=C\beta\abs{x}^\frac{q}{q-1} t^{\beta-1}.
	\end{equation*}
\end{lemma}
\begin{proof}
	The proof is a direct calculation but is included for the convenience of the reader.
	
	Denote $\alpha=\frac{q}{q-1}$. We have
		\begin{equation*}
			\nabla v(x,t)=Ct^{\beta}\alpha\abs{x}^{\alpha-2}x
		\end{equation*}
		and thus
		\begin{align*}
			\Delta_{p}v(x,t)&=\div\left(\abs{\nabla v(x,t)}^{p-2}\nabla v(x,t)\right)\\
			&=\div\left(\abs{Ct^{\beta}\alpha\abs{x}^{\alpha-2}x}^{p-2}Ct^{\beta}\alpha\abs{x}^{\alpha-2}x\right)\\
			&=\left(Ct^{\beta}\alpha\right)^{p-1}\div\left(\abs{x}^{(\alpha-1)(p-2)+\alpha-2}x\right)\\
			&=\left(Ct^{\beta}\alpha\right)^{p-1}\sumn\left(\left[(\alpha-1)(p-2)+\alpha-2\right]\abs{x}^{(\alpha-1)(p-2)+\alpha-4}x_i^2+\abs{x}^{(\alpha-1)(p-2)+\alpha-2}\right)\\
			&=\left(Ct^{\beta}\alpha\right)^{p-1}\left(n+\left[(\alpha-1)(p-1)-1\right]\right)\abs{x}^{(\alpha-1)(p-1)-1}.
		\end{align*}
		It also follows that
		\begin{align*}
			\label{eq:calc1}
			\Delta_{p}^{q}v(x,t)&=\abs{\nabla v(x,t)}^{q-p}\div\left(\abs{\nabla v(x,t)}^{p-2}\nabla v(x,t)\right)\\
			&=\abs{Ct^{\beta}\alpha\abs{x}^{\alpha-2}x}^{q-p}\left(Ct^{\beta}\alpha\right)^{p-1}\left(n+\left[(\alpha-1)(p-2)-\alpha-2\right]\right)\abs{x}^{(\alpha-1)(p-2)+\alpha-2}\\
			&=\left(Ct^{\beta}\alpha\right)^{q-1}\left(n+\left[(\alpha-1)(p-2)+\alpha-2\right]\right)\abs{x}^{(\alpha-1)(p-2)+\alpha-2+(\alpha-1)(q-p)}\\
			&=\left(Ct^{\beta}\alpha\right)^{q-1}\left(n+\left[(\alpha-1)(p-1)-1\right]\right)\abs{x}^{(\alpha-1)(q-1)-1}.\numberthis
		\end{align*}
		Finally
		\begin{equation*}
			(\alpha-1)(q-1)-1=\left(\frac{q}{q-1}-1\right)(q-1)-1=0
		\end{equation*}
		and
		\begin{equation*}
			(\alpha-1)(p-1)-1=\left(\frac{q}{q-1}-1\right)(p-1)-1=\frac{p-q}{q-1}.
		\end{equation*}
		Substituting these and $\alpha$ into \eqref{eq:calc1}, we get exactly what was stated. The time derivative is clear.
\end{proof}
In the next theorem, we prove that regularity of a boundary point is characterized by a barrier family existing at that point. This is our main tool when considering geometric approaches to characterizing regularity.

\begin{theorem}
	\label{thm:regbarequiv}
	Let $\xi_0\in\partial\Theta$. The point $\xi_0$ is regular if and only if there exists a barrier family at $\xi_0$.
\end{theorem}
\begin{proof}
	First, assume that there exists a barrier family at $\xi_0\in\partial\Theta$. Take continuous function $f\in C(\partial\Theta)$. By continuity, for each $\eps>0$ there exists a constant $\delta>0$ such that $\abs{f(\xi)-f(\xi_0)}<\eps$ whenever $\abs{\xi-\xi_0}<\delta$, $\xi\in\partial\Theta$. Thus if $\abs{\xi-\xi_0}<\delta$, we get that
		\begin{equation}
		\label{eq:regbarequiv1}
		f(\xi)-f(\xi_0)-\eps<0\leq \liminf_{\Theta\ni\zeta\to\xi}w_j(\zeta)
		\end{equation}
		because $w_j$ are assumed to be positive. If $\abs{\xi-\xi_0}\geq\delta$, we pick $k\in\N$ such that
		\begin{equation*}
		k>f(\xi)-f(\xi_0)-\eps \quad \text{ and } \quad \delta\geq\frac{1}{k}.
		\end{equation*}
		Now by Definition \ref{def:barfam} condition (c) we know that there exists a $j$ such that
		\begin{equation}
		\label{eq:regbarequiv2}
		\liminf_{\Theta\ni\zeta\to\xi}w_j(\zeta)\geq k>f(\xi)-f(\xi_0)-\eps.
		\end{equation}
		Combining estimates \eqref{eq:regbarequiv1} and \eqref{eq:regbarequiv2}, we get that for some $j\geq1$
		\begin{equation*}
		\liminf_{\Theta\ni\zeta\to\xi}w_j(\zeta)+f(\xi_0)+\eps>f(\xi) \quad\text{ for all }\xi\in\partial\Theta.
		\end{equation*}
		Thus because this is also a supersolution, we have $w_j+f(\xi_0)+\eps\in\mathcal{U}_f$, and hence
		\begin{equation}
		\label{eq:regbarequiv3}
		\limsup_{\Theta\ni\zeta\to\xi}\Hu f(\zeta)\leq\liminf_{\Theta\ni\zeta\to\xi}w_j(\zeta)+f(\xi_0)+\eps=f(\xi_0)+\eps.
		\end{equation}
		By similar calculation as above, $-w_j-\eps+f(\xi_0)\in\mathcal{L}_f$ and we obtain that for some $j$
		\begin{equation}
		\label{eq:regbarequiv4}
		\liminf_{\Theta\ni\zeta\to\xi}\Hu f(\zeta)\geq\liminf_{\Theta\ni\zeta\to\xi}\Hl f(\zeta)\geq -\eps+f(\eps_0),
		\end{equation}
		where the first inequality follows from the fact that $\Hu f\geq\Hl f$ in $\Theta$. This is because we can use the elliptic-type comparison principle Theorem \ref{thm:ecomp} for any pair of $u\in\mathcal{U}_f$ and $v\in\mathcal{L}_f$ in the definitions of Perron solutions. Letting $\eps\to0$, combining \eqref{eq:regbarequiv3} and \eqref{eq:regbarequiv4} gives us that $\xi_0$ is regular.
		
		For the other direction, let us assume that $\xi_0\in\partial\Theta$ is regular. Without loss of generality, we may assume that $\xi_0$ is the origin. For all $(x,t)\in\Rnn$ we define
		\begin{align*}
		\psi_j(x,t)&=j\frac{q-1}{q}\abs{x}^{\frac{q}{q-1}}+j^{q-1}\frac{n+\frac{p-q}{q-1}}{2\diam\Theta}t^2
		\end{align*}
		By Lemma \ref{le:calc}, we have
		\begin{align*}
			\Delta_p^q\psi_j(x,t)&=\left(j\frac{q-1}{q}\frac{q}{q-1}\right)^{q-1}\left[n+\frac{p-q}{q-1}\right]
		\end{align*}
		and
		\begin{equation*}
			\partial_{t}\psi_j(x,t)=j^{q-1}\frac{n+\frac{p-q}{q-1}}{\diam\Theta}t.
		\end{equation*}
		Thus
		\begin{equation*}
		\partial_t\psi_j(x,t)-\Delta_p^q\psi_j(x,t)=j^{q-1}\left[n+\frac{p-q}{q-1}\right]\left(\frac{t}{\diam\Theta}-1\right)\leq 0
		\end{equation*}
		for all $(x,t)\in\Theta$ making $\psi_j$ a subsolution. We will verify that $w_j=\Hl\psi_j$ gives us a barrier family at $\xi_0$ by checking the conditions from Definition \ref{def:barfam}. We have
		\begin{itemize}
			\item[$(a)$:]From $\psi_j\geq0$, it follows also that $w_j\geq0$. Because the set $\Theta$ is bounded, we get
			\begin{equation*}
			\psi_j(x,t)\leq j\frac{q-1}{q}\diam\Theta^{\frac{q}{q-1}}+j^{q-1}\frac{cn}{2}\diam\Theta<\infty
			\end{equation*}
			for every $j$. Thus $w_j$ is a viscosity supersolution by Theorem \ref{thm:perronarevisc}.
			\item[$(b)$:]Follows directly from the regularity of $\xi_0$ because $\psi_j(\xi_0)=0$ for all $j$.
			\item[$(c)$:] Because $\psi_j$ is a viscosity subsolution bounded above by itself on the boundary, we have $\psi_j\in\mathcal{L}_{\psi_j}$ and thus by definition $\Hl\psi_j\geq\psi_j$. Using this, let $k=1,2,\dots,$ and pick $r=\frac{1}{k}$. For any $\xi=(x,t)\in\Theta\setminus B_{r}(\xi_0)$, we have by continuity of $\psi_j$,
			\begin{equation*}
				\liminf_{\Theta\ni\zeta\to\xi}w_j(\zeta)\geq\liminf_{\Theta\ni\zeta\to\xi}\psi_j(\zeta)\geq\psi_j(\xi)\geq j\frac{q-1}{q}r^{\frac{q}{q-1}}+j^{q-1}\frac{n+\frac{p-q}{q-1}}{2\diam\Theta}r^2\geq k,
			\end{equation*}
			where the last inequality holds for large enough $j$. This implies condition $(c)$ of Definition \ref{def:barfam}.\qedhere
		\end{itemize}
\end{proof}
In some cases, the additional conditions satisfied by strong barrier families prove useful in practice. Note that condition (e) gives information over the entire $\overbar{\Theta}$ and implies condition (c), which gives information only over $\partial\Theta$. It turns out that the existence of one type of barrier family implies the other.
\begin{proposition}
	\label{prop:barfam}
	Let $\xi_0\in\partial\Theta$. There exists a barrier family at $\xi_0$ if and only if there exists a strong barrier family at $\xi_0$.
\end{proposition}
\begin{proof}
	A strong barrier family satisfies conditions $(a)-(c)$ by definition and thus is also a barrier family. We will prove the other direction.
	
	Assume that there exists a barrier family at $\xi_0$. By Theorem \ref{thm:regbarequiv}, the point $\xi_0$ is regular. Just as in the proof of Theorem \ref{thm:regbarequiv}, we define
	\begin{align*}
		\psi_j(x,t)&=j\frac{q-1}{q}\abs{x}^{\frac{q}{q-1}}+j^{q-1}\frac{n+\frac{p-q}{q-1}}{2\diam\Theta}t^2
	\end{align*}
	for all $(x,t)\in\Rnn$ and let $w_j=\Hl\psi_j$. We will prove that $w_j$ in fact forms a strong barrier family. In the proof of Theorem \ref{thm:regbarequiv}, we already proved conditions $(a)-(c)$.
	\begin{itemize}
		\item[$(d)$:] Because $w_j$ is a viscosity solution by Theorem \ref{thm:perronarevisc}, it is continuous for every $j$.
		\item[$(e)$:]
		Let
		\begin{equation}
			d(x,t)=\frac{q-1}{q}\abs{x}^{\frac{q}{q-1}}+\frac{n+\frac{p-q}{q-1}}{2\diam\Theta}t^2.
		\end{equation} Notice that $\psi_j(x,t)\geq\min(j,j^{q-1})d(x,t)$. This is continuous and non-negative and $d(x,t)=0$ if and only if $(x,t)=(0,0)$.
		Pick any $k\in\N$. Now by picking $j>\max\{k,k^{\frac{1}{q-1}}\}$, we get
		\begin{equation*}
			w_j=\Hl\psi_j\geq\psi_j\geq\min\{j,j^{q-1}\}d(x,t)\geq kd(x,t)
		\end{equation*}
		as desired. \qedhere
	\end{itemize}
\end{proof}
We get the following restriction result that turns out to be useful in later proofs as a direct corollary of Theorem \ref{thm:regbarequiv}.
\begin{corollary}
	\label{cor:regsubset}
	Let $\xi_0\in\partial\Theta$ and let $G\subset\Theta$ be open and such that $\xi_0\in\partial G$. If $\xi_0$ is regular with respect to $\Theta$, then $\xi_0$ is regular with respect to $G$.
\end{corollary}

\begin{proof}
	Because $\xi_0$ is regular with respect to $\Theta$, we have that by Theorem \ref{thm:regbarequiv} and Proposition \ref{prop:barfam}, there exists a strong barrier family $\{w_j\}_{j=1}^\infty$ in $\Theta$ at point $\xi_0$. Condition $(e)$ from Definition \ref{def:barfam} gives us a non-negative function $d$. Define $\tilde{w}_j=w_j\mid_G$ and $\tilde{d}=d\mid_G$. 
	
	Now $\{\tilde{w}_j\}_{j=1}^\infty$ is a barrier family in $G$ at $\xi_0$ because it clearly satisfies conditions $(a)$, $(b)$ and $(e)$ now with respect to the smaller set $G$ and condition $(e)$ implies $(c)$.
	Thus by using Theorem \ref{thm:regbarequiv} for this barrier family, we have that $\xi_0$ is regular with respect to $G$.
\end{proof}
We will also prove the following proposition to show that regularity is a local property. This is needed later in the proof of the exterior ball condition.
\begin{proposition}
	\label{prop:regsubset}
	Let $\xi_0\in\partial\Theta$ and $B\subset\Rnn$ be any ball containing $\xi_0$. Then $\xi_0$ is regular with respect to $\Theta$ if and only if $\xi_0$ is regular with respect to $\Theta\cap B$.
\end{proposition}
\begin{proof}
	Using Corollary \ref{cor:regsubset}, we know that regularity of $\xi_0$ with respect to $\Theta$ implies regularity with respect to $\Theta\cap B$.
	
	Assume $\xi_0$ is regular with respect to $\Theta\cap B$. By Theorem \ref{thm:regbarequiv} and Proposition \ref{prop:barfam}, there exists a strong barrier family $w_j$ in $\Theta\cap B$. By condition $(e)$ of Definition \ref{def:barfam}, there now exists a non-negative function $d\in C(\overline{\Theta\cap B})$ such that for each $k\in\N$, there exists a $j=j(k)$ such that $w_j\geq kd$ in $\Theta\cap B$.
	Now if we denote $m=\inf_{\overline{\Theta}\cap\partial B}d>0$ and define
	\begin{equation*}
		w_k'(\xi)=\begin{cases}
			\min\{w_{j(k)}(\xi),km\} & \text{ in } \Theta\cap B\\
			km &	\text{ in } \Theta\setminus B
		\end{cases}
	\end{equation*}
	and
	\begin{equation*}
		d'(\xi)=\begin{cases}
			\min\{d(\xi),m\} & \text{ in } \Theta\cap B\\
			m &	\text{ in } \Rnn\setminus B.
		\end{cases}
	\end{equation*}
	Now $w_k'$ is lower semicontinuous, so it is a viscosity supersolution by Lemma \ref{le:pasting}. It also satisfies $w_k'\geq kd'$ in $\Theta$ which can be used to prove remaining condition (c). 
	
	Let $l=1,2,\dots$, and take $\xi\in\partial\Theta$ such that $\abs{\xi-\xi_0}\geq\frac{1}{l}$. Now by picking $k$ large enough to satisfy $kd'(\xi)>l$, we have
	\begin{equation*}
		\liminf_{\Theta\ni\zeta\to\xi}w_j'(\zeta)\geq\liminf_{\Theta\ni\zeta\to\xi}kd'(\zeta)\geq l.
	\end{equation*} Thus $w_k'$ forms a barrier family in $\Theta$. This implies regularity with respect to $\Theta$ by Theorem \ref{thm:regbarequiv}.
\end{proof}
\section{A Counterexample and the multiplied equation}
\label{sec:count}
In this section, we will prove that a single barrier is not enough to guarantee the regularity of a boundary point when $q<2$. We construct a set where the origin is irregular but we can still find a barrier function on that point. We also prove that multiplying one side of \eqref{eq:rgnppar} by a constant does not affect boundary regularity. We need the following scaling lemma for our proof.
\begin{lemma}{(Scaling lemma)}
	\label{le:scaling}
	Let $q\not=2$, $a>0$ and $\Theta\subset\Rnn$ be a domain such that $(0,0)\in\partial\Theta$. Set
	\begin{equation*}
		\tilde{\Theta}=\left\{(ax,t)\in\Rnn:(x,t)\in\Theta\right\}.
	\end{equation*}
	Then $(0,0)$ is regular with respect to $\Theta$ if and only if it is regular to \eqref{eq:rgnppar} with respect to $\tilde{\Theta}$.
\end{lemma}
\begin{proof}
	The proof is almost identical to the $p$-parabolic case in \cite[Proposition 4.1]{Bjorn2017}.
	Take $\tilde{u}:\tilde{\Theta}\to\R$ and define $u:\Theta\to\R$ by
	\begin{equation*}
		u(x,t)=K\tilde{u}(ax,t),
	\end{equation*}
	where $K=a^{-\frac{q}{q-2}}$. By direct calculation $\partial_{t}u(x,t)=K\partial_{t}\tilde{u}(ax,t)$ and
	\begin{align*}
		\Delta_{p}^qu(x,t)&=\abs{\nabla u(x,t)}^{q-p}\Delta_p u\\&=\abs{Ka\nabla \tilde{u}(ax,t)}^{q-p}K^{p-1}a^p\Delta_{p}\tilde{u}(ax,t)\\&=K^{q-1}a^q\abs{\nabla \tilde{u}(ax,t)}^{q-p}\Delta_{p}\tilde{u}(ax,t)\\&=K\Delta_{p}^q\tilde{u}(ax,t).
	\end{align*}
	Thus $u$ is a viscosity supersolution to \eqref{eq:rgnppar} in $\Theta$ if and only if $\tilde{u}$ is a viscosity supersolution to  \eqref{eq:rgnppar} in $\tilde{\Theta}$. Take arbitrary $\tilde{f}\in C(\partial\tilde{\Theta})$ and define $f:\partial\Theta\to\R$ by
	\begin{equation*}
		f(x,t)=K\tilde{f}(ax,t).
	\end{equation*}
	Denote by $\overline{H}_{A}f(x,t)$ the Perron solution defined over set $A$ for bounded $f:\partial A\to\R$. By the calculation above we have
	\begin{equation*}
		\overline{H}_{\Theta}f(x,t)=\overline{H}_{\tilde{\Theta}}(K\tilde{f})(ax,t)
	\end{equation*}
	for all $(x,t)\in\Theta$ and thus regularity of the origin with respect to $\Theta$ implies the same with respect to $\tilde{\Theta}$. Converse is proven by swapping the roles of $\Theta$ and $\tilde{\Theta}$ and replacing $a$ with $\frac{1}{a}$.
\end{proof}
In Theorem \ref{thm:regbarequiv}, we proved that the existence of a barrier family is a sufficient condition for regularity and next, we will prove that the existence of a single barrier is not enough in the singular case. This corresponds to the same result for the $p$-parabolic equation and similarly, the existence of a such counterexample remains open for the degenerate case. The proof is based on constructing suitable boundary values to prove irregularity of the origin and then constructing a barrier at that point.
\begin{theorem}
	\label{thm:counter}
	Let $1<q<2$, $K>0$ and $0<s<\frac{1}{q}$. Then there exists a single barrier $w$ at $(0,0)$ for the domain
	\begin{equation*}
		\Theta=\left\{(x,t)\in\Rn\times\R : \abs{x}\leq K(-t)^s \text{ and } -1<t<0\right\}
	\end{equation*}
	despite $(0,0)$ being irregular.
\end{theorem}
\begin{proof}
	We prove irregularity of $(0,0)$ first. We do this by constructing an explicit viscosity supersolution that is continuous on the boundary but jumps when we approach the origin along the axis $x=0$. This is called an irregularity barrier by \cite{Bjorn2017}, \cite{Kilpelainen1996}, and \cite{Petrovksii1935}. Existence of such function ensures that the boundary point cannot be regular directly from the definition.
	Let
	\begin{equation*}
		u(x,t)=\begin{cases}
			\frac{\abs{x}^{\frac{q}{q-1}}}{(-t)^\frac{qs}{q-1}}-\frac{n+\frac{p-q}{q-1}}{1-qs}\left(\frac{q}{q-1}\right)^{q-1}(-t)^{1-qs}, & \text{ for }(x,t)\in\overline{\Theta}\setminus\{(0,0)\},\\
			1, & \text{ for }(x,t)=(0,0).
		\end{cases}
	\end{equation*}
	Using Lemma \ref{le:calc}, we have
	\begin{align*}
		\Delta_{p}^qu(x,t)&=\left(\frac{q}{q-1}(-t)^{-\frac{qs}{q-1}}\right)^{q-1}\left(n+\frac{p-q}{q-1}\right)=\left(\frac{q}{q-1}\right)^{q-1}\frac{n+\frac{p-q}{q-1}}{(-t)^{qs}}.
	\end{align*}
	and by direct calculation
	\begin{equation*}
		\partial_t u(x,t)=\frac{qs}{q-1}\frac{\abs{x}^{\frac{q}{q-1}}}{(-t)^{\frac{qs}{q-1}+1}}+\left(\frac{q}{q-1}\right)^{q-1}\frac{n+\frac{p-q}{q-1}}{(-t)^{qs}}.
	\end{equation*}
	Thus $\partial_{t}u\geq\Delta_{p}^qu$ and hence $u$ is viscosity supersolution to \eqref{eq:rgnppar} in $\Theta$.
	
	Let $f=u\mid_{\partial\Theta}\in C(\partial\Theta)$ and let $v\in\mathcal{L}_f(\Theta)$. By definition of the lower class, we can use the elliptic-type comparison principle Theorem \ref{thm:ecomp} to ensure $v\leq u$ in $\Theta$, and thus also $\underline{H}f\leq u$. But now
	\begin{equation*}
		\liminf_{\Theta \ni(x, t) \rightarrow(0, 0)}\underline{H}f(x,t)\leq\liminf_{\Theta \ni(x, t) \rightarrow(0, 0)}u(x,t)\leq\liminf_{t\to0-}u(0,t)=0<1=f(0,0).
	\end{equation*}
	Hence $(0,0)$ is irregular to equation \eqref{eq:rgnppar} with respect to the set $\Theta$. 
	
	Next, we will show that there still exists a barrier at $(0,0)$. Assume first that $K=1$ and let
	\begin{equation*}
		v(x,t)=(-t)^\frac{1}{2-q}\left(B-\abs{x}^{\frac{q}{q-1}}\right)
	\end{equation*}
	where $B=\min\left\{\left(n+\frac{p-q}{q-1}\right)\left(\frac{q}{q-1}\right)^{q-1}(2-q),1\right\}$.
	Again by Lemma \ref{le:calc}, we have
	\begin{align*}
		\Delta_{p}^qv(x,t)&=-\left((-t)^{\frac{1}{2-q}}\frac{q}{q-1}\right)^{q-1}\left(n+\frac{p-q}{q-1}\right)\\&=-(-t)^{\frac{q-1}{2-q}}\left(\frac{q}{q-1}\right)^{q-1}\left(n+\frac{p-q}{q-1}\right).
	\end{align*}
	For the time derivative, we have
	\begin{equation*}
		\partial_{t}v(x,t)=-\frac{1}{2-q}(-t)^{\frac{q-1}{2-q}}\left(B-\abs{x}^{\frac{q}{q-1}}\right)\geq-\frac{B}{2-q}(-t)^{\frac{q-1}{2-q}}
	\end{equation*}
	and thus by our choice of $B$, we have
	\begin{equation*}
		\partial_{t}v(x,t)-\Delta_{p}^qv(x,t)\geq\left(\left(\frac{q}{q-1}\right)^{q-1}\left(n+\frac{p-q}{q-1}\right)-\frac{B}{2-q}\right)(-t)^{\frac{q-1}{2-q}}\geq0
	\end{equation*} and thus $v$ is a viscosity supersolution to \eqref{eq:rgnppar} in $\Theta$. Next, we define
	\begin{equation*}
		\tilde{\Theta}=\left\{(x,t)\in\Theta:\abs{x}^{\frac{q}{q-1}}<\frac{B}{2}\right\}
	\end{equation*}
	and
	\begin{equation*}
		M=\inf_{(x,t)\in\partial\tilde{\Theta}}v(x,t)=\left(\frac{B}{2}\right)^{1+\frac{q-1}{sq(2-q)}}>0.
	\end{equation*}
	By the pasting lemma (Lemma \ref{le:pasting}), we know that
	\begin{equation*}
		w(x,t)=\begin{cases}
			\min\{v(x,t),M\} & \text{ if }(x,t)\in\tilde{\Theta},\\
			M & \text{ if }(x,t)\in\Theta\setminus\tilde{\Theta},
		\end{cases}
	\end{equation*}
	is a viscosity supersolution to \eqref{eq:rgnppar} in $\Theta$ because it is lower semicontinuous. It also clearly satisfies the other two conditions of being a barrier and thus we have found a barrier at $(0,0)$ despite this point being irregular. The result for general $K>0$ follows from Lemma \ref{le:scaling}.
\end{proof}
If we take a viscosity solution $u$ to equation \eqref{eq:rgnppar} and a constant $c>0$, a simple calculation shows that the function $cu$ is not a viscosity solution unless $q=2$. This also happens for the usual $p$-parabolic equation with $p\not=2$. We get similar phenomena where $cu$ now solves the multiplied equation
\begin{equation}
	\label{eq:mult}
	a\partial_{t}u=\Delta_{p}^qu.
\end{equation}
for $a=c^{q-2}$. It quite surprisingly turns out that regular boundary points are the same for all multiplied equations of this type as long as $q\not=2$ which we will prove next. This is known to be false for the heat equation by the Petrovski\u{\i} condition, see \cite{Petrovksii1935}.
\begin{theorem}
	Let $\xi_0\in\partial\Theta$ and $a>0$. If $q\not=2$, the $\xi_0$ is regular if and only if it is regular to the multiplied equation \eqref{eq:mult}.
\end{theorem}
\begin{proof}
	Let $w$ be a viscosity supersolution to \eqref{eq:rgnppar} and let $\tilde{w}=a^{\frac{1}{q-2}}w$. Then
	\begin{equation*}
		a\partial_t\tilde{w}-\Delta_{p}^q\tilde{w}=a^{1+\frac{1}{q-2}}\partial_{t}w-a^{1+\frac{1}{q-2}}\Delta_{p}^qw\geq 0
	\end{equation*}
	and thus $\tilde{w}$ is a viscosity supersolution to the multiplied equation \eqref{eq:mult}. We get equivalence by replacing $a$ by $a^{-1}$.
	
	From this it follows that if $u\in\mathcal{U}_f$ if and only if $a^{\frac{1}{q-2}}u\in\mathcal{U}_f^a$, where $\mathcal{U}_f^a$ is the upper class with respect to equation \eqref{eq:mult}. The equivalence of regularity of $\xi_0$ with respect to equation \eqref{eq:rgnppar} and with respect to equation \eqref{eq:mult} follow directly from the definition.
\end{proof}
When $q=2$, our equation becomes the normalized $p$-parabolic equation, and we know that a single barrier is enough to characterize the regularity of a boundary point as proven by \cite[Theorem 4.2]{Banerjee2014}. This seems to be the case because invariance with regard to multiplication means that the existence of a single barrier implies the existence of a barrier family. We will prove this result to end this section. Based on this it would seem that a single barrier is not enough when $q\not=2$. We know this to be the case for $q<2$ by Theorem \ref{thm:counter} but the degenerate case $q>2$ remains an open problem.
\begin{proposition}
	Let $\xi_0\in\partial\Theta$ and $q=2$. There exists a barrier at $\xi_0$ if and only if there exists a barrier family at $\xi_0$.
\end{proposition}
\begin{proof}
	The existence of a barrier family clearly implies the existence of a single barrier. For the other direction assume that $w$ is a barrier at $\xi_0$ and define
	\begin{equation*}
		w_j=jw
	\end{equation*}
	for $j\in\N$. Now $w_j$ are all positive viscosity supersolutions to $\eqref{eq:rgnppar}$ by simple calculation because $q=2$ and still clearly satisfy condition $(b)$ for all $j$. Take $k\in\N$ and $\xi\in\partial\Theta$ such that $\abs{\xi-\xi_0}\geq\frac{1}{k}$. Because $w$ is a barrier, we know that
	\begin{equation*}
		\liminf_{\Theta\ni\zeta\to\xi} w(\zeta)=a>0
	\end{equation*}
	and thus by choosing $j>\frac{k}{a}$, we have
	\begin{equation*}
		\liminf_{\Theta\ni\zeta\to\xi} w_j(\zeta)=k.
	\end{equation*}
	This proves condition $(c)$ and thus we have a barrier family.
\end{proof}
\section{Exterior ball condition}
\label{sec:ext}
In this section, we will state and prove an exterior ball condition, which gives a simple geometric criterion for regularity. It turns out that the existence of an exterior ball touching the domain at a boundary point implies the existence of a suitable barrier family apart from a few exceptions. Consider the Dirichlet problem \eqref{eq:dirichlet} set in the usual time cylinder $\Om_T$. It is well known that in this case, the solution will determine values on the top of the cylinder $\Om\times\{T\}$, so none of these points can be regular. It turns out that if the tangent point of the exterior ball is its north pole or south pole the argument does not work. After this, we prove a different geometric condition that works for the north pole case and end the section by showing that any point that is time-wise earliest in the set, is always regular.
\begin{lemma}[Exterior ball condition]
	Let $\xi_0=(x_0,t_0)\in\partial\Theta$. Suppose that there exists a $\xi_1=(x_1,t_1)\in\Theta^c$ and a radii $R_1>0$ such that $B_{R_1}(\xi_1)\cap\Theta=\emptyset$ and $\xi_0\in \partial B_{R_1}(\xi_1)\cap\partial\Theta$. If $x_1\not=x_0$ then $\xi_0$ is regular with respect to $\Theta$.
	
%	, or if $\xi_0$ is the north pole of $B_{R_1}(\xi_1)$ and the additional radius condition $R_1>n+p-2$ is satisfied, 
\end{lemma}
\begin{proof}
	The case $q=2$ is proven in \cite[Lemma 4.2]{Bjorn2019}, so we may assume $q\not=2$. Without loss of generality, assume $\xi_0=(0,0)$ and $\partial B_{R_1}(\xi_1)\cap\partial\Theta=\{\xi_0\}$. Let $\xi_2=(x_2,t_2)=\frac12\xi_1$ and $R_2=\frac12R_1$. Also pick $\delta=\frac12\abs{x_2}>0$ and let $\Theta_0=\Theta\cap B_{\delta}(\xi_0)$. Positivity of $\delta$ follows from our assumption $x_1\not=x_0$. For $\xi=(x,t)\in\overline{\Theta}_0$ and $R=\abs{\xi-\xi_2}\leq 2R_2$, define
	\begin{equation*}
		w_j(\xi)=\gamma \left(e^{-jR_2^2}-e^{-jR^2}\right)
	\end{equation*}
	where we will choose $j$ and $\gamma=\gamma(j)>0$ later. We will show that for suitable constants, $w_j$ are a barrier family at $\xi_0$. We have
	\begin{align*}
		\partial_t w_j(\xi)&=2j\gamma e^{-jR^2}(t-t_2)\geq-4j\gamma R_2e^{-jR^2},\\
		\nabla w_j(\xi)&=2j\gamma e^{-jR^2}(x-x_2),
	\end{align*}
	which means that, $\xi^i$ denoting the $i$:th coordinate of vector $\xi$,
	\begin{align*}
		\Delta_{p}w_j(\xi)&=\div\left(\abs{2j\gamma e^{-jR^2}(x-x_2)}^{p-2}2j\gamma e^{-jR^2}(x-x_2)\right)\\
		&=(2j\gamma)^{p-1}\sumn\left[\left(-2j(p-1)e^{-j(p-1)R^2}\abs{x-x_2}^{p-2}(x^i-x_2^i)^2\right)\right.\\&\left.\quad +\left((p-2)e^{-j(p-1)R^2}\abs{x-x_2}^{p-4}(x^i-x_2^i)^2\right)+\left(e^{-j(p-1)R^2}\abs{x-x_2}^{p-2}\right)\right]\\
		&=(2j\gamma)^{p-1}e^{-j(p-1)R^2}\abs{x-x_2}^{p-2}\left[-2j(p-1)\abs{x-x_2}^2 +p-2+n\right]
	\end{align*}
	and further
	\begin{align*}
		\Delta_{p}^qw_j(\xi)&=\abs{2j\gamma e^{-jR^2}(x-x_2)}^{q-p}\Delta_{p}w_j(\xi)\\
		&=(2j\gamma)^{q-1}e^{-j(q-1)R^2}\abs{x-x_2}^{q-2}\left[-2j(p-1)\abs{x-x_2}^2 +p-2+n\right]\\
		&\leq(2j\gamma)^{q-1}e^{-j(q-1)R^2}\abs{x-x_2}^{q-2}\left[-2j(p-1)\delta^2 +p-2+n\right],
	\end{align*}
	because $\abs{x-x_2}\geq\delta$. Now choose $j_0>\frac{p-2+n}{(p-1)\delta^2}$ to be an integer, so that for any $j\geq j_0$, we have
	\begin{equation*}
		-2j(p-1)\delta^2 +p-2+n\leq-j(p-1)\delta^2,
	\end{equation*}
	which implies
	\begin{align*}
		\Delta_{p}^qw_j(\xi)&\leq-(2j\gamma)^{q-1}j(p-1)\delta^2e^{-j(q-1)R^2}\abs{x-x_2}^{q-2}\\
		&\leq-C_0(2j\gamma)^{q-1}j(p-1)e^{-j(q-1)R^2},
	\end{align*}
	where
	\begin{equation*}
		C_0=\begin{cases}
			(2R_2)^{q-2}\delta^2, & 1<q<2,\\
			\delta^q, & q>2.
		\end{cases}
	\end{equation*}
	Based on these calculations we have that $w_j$ is a viscosity supersolution to \eqref{eq:rgnppar} for all $j\geq j_0$ if
	\begin{equation*}
		4j\gamma R_2e^{-jR^2}\leq C_0(2j\gamma)^{q-1}j(p-1)e^{-j(q-1)R^2},
	\end{equation*}
	which is equivalent to
	\begin{equation}
		\label{eq:exball1}
		\gamma^{q-2}\geq\frac{j^{1-q}R_2e^{j(q-2)R^2}}{2^{q-3}(q-1)C_0}=C_1j^{1-q}e^{j(q-2)R^2}
	\end{equation}
	for $C_1=\frac{R_2}{2^{q-3}(q-1)C_0}$. Now we choose
	\begin{equation}
		\label{eq:exball2}
		\gamma=\gamma(j)=\begin{cases}
			\left(C_1j^{1-q}\right)^{\frac{1}{q-2}}e^{jR_2^2}, & 1<q<2,\\
			\left(C_1j^{1-q}\right)^{\frac{1}{q-2}}e^{4jR_2^2}, & q>2.
		\end{cases}
	\end{equation}
	Because $B_{R_2}(\xi_2)\cap\Theta_0$ is empty, we necessarily have $R\in(R_2,2R_2)$ and thus for this $\gamma$, the estimate \eqref{eq:exball1} holds and thus $w_j$ is a positive viscosity supersolution to \eqref{eq:rgnppar} for all $j\geq j_0$. 
	
	We still need to check the rest of the conditions of Definition \ref{def:barfam} to ensure that $w_j$ forms a barrier family. Condition $(b)$ clearly holds by continuity of $w_j$ at $\xi_0$. For condition $(c)$, let $\beta$ be the angle between vectors $-\xi_1$ and $\xi-\xi_1$ and denote $r_0=\abs{\xi}$ and $r_1=\abs{\xi-\xi_1}$. Using the cosine theorem, we get the equalities
	\begin{equation*}
		R^2=r_1^2+\left(\frac{R_1}{2}\right)^2-r_1R_1\cos\beta \quad \text{ and } \quad r_0^2=r_1^2+R_1^2-2r_1R_1\cos\beta.
	\end{equation*}
	Using these one after another and lastly the inequality $r_1\geq R_1$, we get
	\begin{align*}
		R^2-R_2^2&=\left(\frac{R_1}{2}\right)^2+r_1^2-r_1R_1\cos\beta-\left(\frac{R_1}{2}\right)^2\\
		&=r_1^2-\frac{1}{2}(r_1^2+R_1^2-r_0^2)=\frac{1}{2}r_1^2-\frac{R_1}{2}^2+\frac{1}{2}r_0^2\\
		&\geq\frac{1}{2}r_0^2.
	\end{align*}
Using this we can estimate the value of the barrier function. For any $r>0$ and $\xi\in\overline{\Theta}_0\setminus B_r(x_0)$,
\begin{equation*}
	w_j(\xi)=\gamma e^{-jR_2^2}(1-e^{j(R_2^2-R^2)})\geq\gamma e^{-jR_2^2}(1-e^{-\frac{j}{2}r^2}).
\end{equation*}
Inserting our choices of $\gamma$ from equation \eqref{eq:exball2}, we have an estimate
\begin{equation}
	\label{eq:exball3}
	w_j(\xi)\geq\begin{cases}
		\left(C_1j^{1-q}\right)^{\frac{1}{q-2}}(1-e^{-\frac{j}{2}r^2}), & 1<q<2,\\
		\left(C_1j^{1-q}\right)^{\frac{1}{q-2}}e^{3jR_2^2}(1-e^{-\frac{j}{2}r^2}), & q>2.
	\end{cases}
\end{equation}
In either case, the right-hand side tends to $\infty$ as $j\to\infty$ for any fixed $r$. Thus for any $k\in\N$ we can pick $r=\frac{1}{k}$ and equation \eqref{eq:exball3} implies
\begin{equation*}
	\liminf_{\Theta\ni\zeta\to\xi}w_j(\zeta)\geq k
\end{equation*}
for some large $j$ and any $\xi\in\overline{\Theta}_0\setminus B_{r}(x_0)$. Thus condition $(c)$ holds and $w_j$ forms a barrier family. Thus by Theorem \ref{thm:regbarequiv} the point $\xi_0$ is regular to equation \eqref{eq:rgnppar} with respect to the set $\Theta_0$. The regularity with respect to the set $\Theta$ follows from Proposition \ref{prop:regsubset}.
\end{proof}
Similarly to the $p$-parabolic case, this proof does not work when $\xi_0$ is the north pole or the south pole of the ball. In the north pole case, we get a result matching to \cite[Proposition 4.2]{Bjorn2015}.
\begin{proposition}
	\label{prop:north}
	Let $\Theta\in\Rnn$ be open and $(x_0,t_0)\in\partial\Theta$. Assume that for some $\theta>0$, we have
	\begin{equation*}
		\Theta\subset\{(x,t)\mid t-t_0>-\theta\abs{x-x_0}^l\},
	\end{equation*}
	where $l\geq\frac{q}{q-1}$ if $1<q<2$, and $l>q$ if $q>2$. Then $(x_0,t_0)$ is regular with respect to $\Theta$.
\end{proposition}
\begin{proof}
	Without loss of generality, we may assume $(x_0,t_0)=(0,0)$.
	Let
	\begin{equation*}
		\Theta_0=\left\{(x,t)\mid t>-\theta\abs{x}^l \text{ and } -1<t<0 \right\}
	\end{equation*}
	and
	\begin{equation*}
		G^j=\left\{(x,t)\in\Theta_0\,\Bigg|\, \abs{x}<\frac{1}{j^{\frac{1}{s}}} \text{ and } -\frac{\theta}{j^{\frac{l}{s}}}<t<0\right\},
	\end{equation*}
	where $s>0$ will be fixed later. Note that $G_{j+1}\subset G_j$ for all $j$.
	Now let
	\begin{equation*}
		f_j(x,t)=j\frac{q-1}{q}\abs{x}^{\frac{q}{q-1}}+\left(n+\frac{p-q}{q-1}\right)j^{q-1}t
	\end{equation*}
	and use Lemma \ref{le:calc} to conclude that
	\begin{equation*}
		\Delta_{p}^qf_j(x,t)=\left(j\frac{q-1}{q}\frac{q}{q-1}\right)^{q-1}\left(n+\frac{p-q}{q-1}\right)=j^{q-1}\left(n+\frac{p-q}{q-1}\right)
	\end{equation*}
	and
	\begin{equation*}
		\partial_tf_j(x,t)=\left(n+\frac{p-q}{q-1}\right)j^{q-1}.
	\end{equation*}
	We can see that $f_j$ is a viscosity solution to $\eqref{eq:rgnppar}$ in $\Rnn$.
	Define
	\begin{align*}
		m_j:&=\inf_{\Theta\cap\partial G^j}f_j=j\frac{q-1}{q}\left(\frac{1}{j^\frac{1}{s}}\right)^{\frac{q}{q-1}}-\left(n+\frac{p-q}{q-1}\right)j^{q-1}\frac{\theta}{j^{\frac{l}{s}}}\\
		&=\frac{q-1}{q}j^{1-\frac{q}{s(q-1)}}-\left(n+\frac{p-q}{q-1}\right)\theta j^{q-1-\frac{l}{s}}.
	\end{align*}
	We want $m_j\to\infty$ as $j\to\infty$ to construct a barrier family. Note that the coefficient of the second term is always positive so we only need to take care of the exponents. We must have
	\begin{equation*}
		1-\frac{q}{s(q-1)}>0 \quad \text{ and } \quad 1-\frac{q}{s(q-1)}>q-1-\frac{l}{s}
	\end{equation*}
	i.e.
	\begin{equation}
		\label{eq:north1}
		s>\frac{q}{q-1} \quad \text{ and } \quad l>s(q-2)+\frac{q}{(q-1)}.
	\end{equation}
	The latter condition gives us the two cases depending on $q$. If $1<q<2$, the first term on the right-hand side is negative and we have
	\begin{equation*}
		l>\frac{q}{q-1} \quad \text{ for all } \quad s>\frac{q}{q-1}.
	\end{equation*}
	If $q>2$, the first term on the right-hand side is positive and we have $l>q$ if we choose $s$ sufficiently close to $\frac{q}{q-1}$. But we see that there exists a $s$ that is suitable for both of these cases.
	
	Now we define
	\begin{equation*}
		h_j=\begin{cases}
			\min\{f_j,m_j\} & \text{ in } G^j \\
			m_j & \text{ in } \Theta_0\setminus G^j.
		\end{cases}
	\end{equation*}
	Now using the pasting lemma (Lemma \ref{le:pasting}) for $f_j\mid_{G^j}$ and $m_j$, we have that $h_j$ is a positive viscosity supersolution to \eqref{eq:rgnppar} in $\Theta_0$ provided we choose $j$ large enough. Conditions $(a)$ and $(b)$ of the definition of a barrier Definition \ref{def:barfam} are clearly satisfied. 
	
	Now for these $s$ and $l$ satisfying \eqref{eq:north1}, we have $m_j\to\infty$ and $\abs{\xi}\to0$ for all $\xi\in G_j$ as $j\to\infty$ and thus for any $k$, we are able to find some large $j(k)$ so condition $(c)$ is satisfied. The family of functions $h_j$ is thus a barrier family at $(0,0)$. Thus by Theorem \ref{thm:regbarequiv}, the point $(0,0)$ is regular with respect to $\Theta_0$.
	
	Regularity with respect to $\Theta$ follows by picking an open ball $B$ containing $(0,0)$, using Corollary \ref{cor:regsubset} for $\Theta_0\cap B$ and then Proposition \ref{prop:regsubset}.
\end{proof}
Finally, if the bottom of the set is flat, we get regularity for all of these points from the following useful lemma. It turns out that the earliest points time-wise are always regular.
\begin{lemma}
	Let $\xi_0=(x_0,t_0)\in\partial\Theta$. If $\xi_0\not\in\partial\Theta_-$ for
	\begin{equation*}
		\Theta_-=\{(x,t)\in\Theta\mid t<t_0\},
	\end{equation*}
	then $\xi_0$ is regular with respect to $\Theta$. In particular, this holds if $\Theta_-=\emptyset$.
\end{lemma}
\begin{proof}
	Let
	\begin{equation*}
		f_j(x,t)=j\frac{q-1}{q}\abs{x-x_0}^{\frac{q}{q-1}}+\left(n+\frac{p-q}{q-1}\right)j^{q-1}(t-t_0).
	\end{equation*}
	For any $(x,t)\not\in\partial\Theta_-$, we can use Lemma \ref{le:calc} to conclude
	\begin{equation*}
		\Delta_{p}^qf_j(x,t)=j^{q-1}\left(n+\frac{p-q}{q-1}\right)=\partial_tf_j(x,t),
	\end{equation*}
	which implies that $f_j$ are positive viscosity solutions to \eqref{eq:rgnppar} in $\Rnn$ for all $j$. They also clearly satisfy condition $(b)$ of Definition \ref{def:barfam} and lastly for any $k=1,2,\dots$, $r=\frac{1}{k}$ and $\xi\in\Theta\setminus B_r(\xi_0)$, we can ensure
	\begin{equation*}
		\liminf_{\Theta\ni\zeta\to\xi}f_j(\zeta)\geq j\frac{q-1}{q}r^{\frac{q}{q-1}}+\left(n+\frac{p-q}{q-1}\right)j^{q-1}r\geq k
	\end{equation*}
	by picking a large $j$, which implies condition $(c)$. Thus $f_j$ form a barrier family in $\Theta$ at a point $\xi_0$ and thus by Theorem \ref{thm:regbarequiv}, the point $\xi_0$ is regular with respect to $\Theta$.
\end{proof}
There remain many other geometric conditions known for the $p$-parabolic equation that could be expanded for equation \eqref{eq:rgnppar} in the future.
\section{Superparabolic}
\label{sec:app}
Our definition of barriers differs from what was used by Björn, Björn, Gianazza, and Parviainen in their paper, and we will prove in this final section that the definitions coincide. The definition is otherwise the same but they assume the function satisfies the comparison principle in arbitrary time cylinders instead of directly defining them to be viscosity supersolutions. This type of function has various names in the literature for different equations. For the usual $p$-parabolic equation these are sometimes called $p$-superparabolic in the literature and generalized supersolutions for the normalized equation in \cite{Banerjee2014}. We will just use the term superparabolic for simplicity.
\begin{definition}
	A function $u:\Theta\to(-\infty,\infty]$ is superparabolic to equation \eqref{eq:rgnppar} in $\Theta$ if
	\begin{itemize}
		\item [(i)]$u$ is lower semicontinuous,
		\item [(ii)]$u$ is finite in a dense subset of $\Theta$,
		\item [(iii)]$u$ satisfies the following comparison principle on each space-time cylinder $\Om_{t_1,t_2}$: If $v\in C(\overline{\Om}_{t_1,t_2})$ is a viscosity solution to \eqref{eq:rgnppar} in $\Om_{t_1,t_2}$ satisfying $v\leq u$ on $\partial_{p}\Om_{t_1,t_2}$, then $v\leq u$ in $\Om_{t_1,t_2}$.
	\end{itemize}
\end{definition}
Superparabolic functions defined in this way for $p$-parabolic and normalized $p$-parabolic equations are the same as the corresponding viscosity solutions as shown in \cite{Banerjee2014} and \cite{Juutinen2001}. We will prove this same result for equation \eqref{eq:rgnppar}.
\begin{lemma}
	\label{le:barrierdef}
	In a given domain, the viscosity supersolutions and superparabolic functions to \eqref{eq:rgnppar} are the same.
\end{lemma}
\begin{proof}
	A viscosity supersolution is clearly superparabolic because conditions (i) and (ii) already match and supersolution satisfies the comparison principle Theorem \ref{thm:comp}. 
	
	Let $\Theta\subset\Rnn$ and $u$ be a superparabolic to \eqref{eq:rgnppar} in $\Theta$. We assume thriving for a contradiction that $u$ is not a viscosity supersolution in $\Theta$. Then we must have a point $(x_0,t_0)\in\Theta$ and at least one admissible $\vp\in C^2(\Theta)$ that touches $u$ at $(x_0,t_0)$ from below but we have one of the following cases.
	
	\textbf{Case 1:} $\partial_{t}\vp(x_0,t_0)-\Delta_{p}^q\vp(x_0,t_0)<0$ and $\nabla\vp(x_0,t_0)\not=0$. Because the inequality is strict, we necessarily have that for small $\rho>0$ the function $\vp$ is a classical subsolution to \eqref{eq:rgnppar} inside a small cylinder $Q_\rho$. By the definition of touching from below, it is possible for us to choose $\rho$ so small that we can pick $\delta>0$ so that \begin{equation*}
		\vp+\delta\leq u \text{ on } \partial_{p}Q_\rho.
	\end{equation*} Let $v$ be a viscosity solution to the Dirichlet problem \eqref{eq:dirichlet} with $\vp+\delta$ as boundary values on the set $Q_\rho$. This exists by Theorem \ref{thm:existball}. Because $\vp+\delta$ is a subsolution in this set, we can use the comparison principle Theorem \ref{thm:comp} to deduce $\vp+\delta\leq v$ in $Q_\rho$. This combined with condition (iii) from $u$ being superparabolic gives us 
	\begin{equation*}
		\vp(x,t)+\delta\leq v(x,t)\leq u(x,t) \text{ for all }(x,t)\in Q_\rho.
	\end{equation*}
	But this is a contradiction as this implies $\vp(x_0,t_0)+\delta\leq u(x_0,t_0)=\vp(x_0,t_0)$.
	
	%By Remark 2.2.7. from \cite{Giga2006}, we may assume that the gradient $\nabla\vp$ is non-zero in some small punctured neighborhood $V^*$ of $(x_0,t_0)$.
	\textbf{Case 2:} $\partial_{t}\vp(x_0,t_0)<0$ and $\nabla\vp(x_0,t_0)=0$. Because $\vp$ is admissible, we have by Definition \ref{def:admissible}, that for some $\rho>0$, $f\in\mathcal{F}(F)$ and $g(x)=f(\abs{x})$, we have
	\begin{equation}
		\label{eq:barrierdef1}
		\abs{\vp(x,t)-\vp(x_0,t_0)-\partial_t\vp(x_0,t_0)(t-t_0)}\leq g(x-x_0)+\sigma(t-t_0)
	\end{equation}
	for all $(x,t)\in B_\rho(x_0)\times(t_0-\rho,t_0+\rho)$. 
	Define
	\begin{equation*}
		\phi(x,t)=u(x_0,t_0)+\partial_t\vp(x_0,t_0)(t-t_0)-g(x-x_0)-\sigma(t-t_0)
	\end{equation*}
	which is an admissible test function touching $u$ at $(x_0,t_0)$ from below because
	\begin{align*}
		\abs{\phi(x,t)-\phi(x_0,t_0)-\partial_t\phi(x_0,t_0)(t-t_0)}&=\abs{\phi(x,t)-u(x_0,t_0)-\partial_t\vp(x_0,t_0)(t-t_0)}
		\\&\leq g(x-x_0)+\sigma(t-t_0).
	\end{align*}
	By definition of $\mathcal{F}(F)$, $g$ satisfies
	\begin{equation}
		\label{eq:barrierdef2}
		\lim_{\substack{x\to0\\x\not=0}}\Delta_p^qg(x-x_0)=0.
	\end{equation}
	Because we assumed that $\partial_t\vp(x_0,t_0)<0$ and $\sigma'(t-t_0)>0$ in some small punctured neighborhood of $t_0$ by definition of $\Sigma$, a direct calculation using \eqref{eq:barrierdef2} yields
	\begin{equation*}
		\partial_t\phi(x,t)-\Delta_p^q\phi(x,t)=\partial_t\vp(x_0,t_0)-\sigma'(t-t_0)-\Delta_p^qg(x-x_0)<0
	\end{equation*}
	in some small punctured neighborhood $V^*$ of $(x_0,t_0)$. This means that $\phi$ is a classical subsolution to \eqref{eq:rgnppar} in $V^*$. Because we have $\partial_{t}\phi(x_0,t_0)<0$, we can conclude by \cite[Lemma 4.1]{Ohnuma1997} that $\phi$ is a viscosity subsolution to \eqref{eq:rgnppar} in $V=V^*\cup\{(x_0,t_0)\}$.
	
	The rest is similar to Case 1. By the definition of touching from below, it is possible for us to choose $\tilde{\rho}$ so small that we can pick $\delta>0$ so that \begin{equation*}
		\phi+\delta\leq u \text{ on } \partial_{p}Q_{\tilde{\rho}}
	\end{equation*}
	and $Q_{\tilde{\rho}}\subset V$. Let $v$ be a viscosity solution to the Dirichlet problem \eqref{eq:dirichlet} with $\phi+\delta$ as boundary values on the set $Q_{\tilde{\rho}}$. This exists by Theorem \ref{thm:existball}. Because $\phi+\delta$ is a viscosity subsolution in this set, we can use the comparison principle Theorem \ref{thm:comp} to deduce $\phi+\delta\leq v$ in $Q_{\tilde{\rho}}$. This combined with condition (iii) from $u$ being superparabolic gives us 
	\begin{equation*}
		\phi(x,t)+\delta\leq v(x,t)\leq u(x,t) \text{ for all }(x,t)\in Q_{\tilde{\rho}}.
	\end{equation*}
	But this is a contradiction as this implies $\phi(x_0,t_0)+\delta\leq u(x_0,t_0)=\phi(x_0,t_0)$.
	
	Both cases lead to a contradiction and thus $u$ is a viscosity supersolution.
\end{proof}

%\bibliographystyle{alphaabbr}
%\bibliography{../../../bibtex/boundary}

\end{document}